\documentclass[11pt]{amsart}

\usepackage[all]{xy}
\usepackage{amsmath, amssymb, amsfonts, hyperref, color}
\usepackage{epsfig}
\usepackage{colordvi}
\usepackage{mathrsfs} 

\newcommand{\pp}[1]{{\mathfrak #1}}
\newcommand{\proj}[1]{\operatorname{Proj}(#1)}
\newcommand{\length}{\operatorname{length}}
\newcommand{\spec}{\operatorname{Spec}}

\newcommand{\rank}{\operatorname{rank}}
\newcommand{\rk}{\operatorname{rk}}

\newcommand{\divs}{\operatorname{div}}

\newcommand{\tor}{\operatorname{Tor}}

\newcommand{\cl}{\operatorname{Cl}}
\newcommand{\ch}{\operatorname{ch}}
\newcommand{\td}{\operatorname{td}}

\newcommand{\charact}{\operatorname{char}}

\newcommand{\HS}{\pp h} 
\newcommand{\HSP}{\pp p} 
\newcommand{\hk}[2]{\varphi_{#1}(#2)} 
\newcommand{\hkq}[2]{\Phi_{#1}(#2)} 
\newcommand{\hks}[1]{\operatorname{HKS}_{#1}(t)} 

\newcommand{\gp}[1]{\mathbb Z {#1}} 
\newcommand{\Ff}{ {\bf {\sc F}}} 

\theoremstyle{plain}
\newtheorem{theorem}{Theorem}[section]

\newtheorem{lemma}[theorem]{Lemma}

\newtheorem{corollary}[theorem]{Corollary}

{\Alph{theorem}}

\theoremstyle{definition}

\newtheorem{example}[theorem]{Example}
\newtheorem{application}[theorem]{Application}
\newtheorem{chunk}[theorem]{}

\newtheorem{remark}[theorem]{Remark}

\theoremstyle{remark}
\newtheorem{question}[theorem]{Question}

\numberwithin{equation}{section}

\title[Hilbert-Kunz Functions] 
{ The Shape of Hilbert-Kunz Functions}
\author{C-Y. Jean Chan }
\begin{document}

\maketitle

\begin{abstract}

We discuss Hilbert-Kunz function from when it was originally defined to its recent developments. 
A brief history of Hilbert-Kunz theory is first recounted. Then we review several techniques involved in 
the study of Hilbert-Kunz functions by presenting some illustrative proofs without going into  
details of the technicalities.

The second part of this article focuses on the Hilbert-Kunz function of an affine normal semigroup ring and relates it to 
Ehrhart quasipolynomials. We pay extra attention to its periodic behavior and discuss how 
the cellular decomposition constructed by Bruns and Gubeladze fits into the computation of the functions. 
The closed forms of the Hilbert-Kunz function of some examples are presented.
The discussion in this part highlights the close relationship between Hilbert-Kunz function and  Ehrhart theory.

\end{abstract}

\bigskip


\section{Motivation and Outline}\label{motivation} 

In the 1960s, Kunz~\cite{Kun} introduced a function in order to study the regularity of integral domains. 
Monsky named this function after Hilbert and Kunz 
in \cite[1983]{Mon1}. Despite its close resemblance with the usual Hilbert-Samuel functions, Hilbert-Kunz functions, in many ways, behave very differently and are highly unpredictable. 

The main aim of this article is to provide an overview of the development of Hilbert-Kunz functions in the recent decades, and 
to link Hilbert-Kunz theory to Ehrhart theory that may potentially provide accessible tools to investigate 
Hilbert-Kunz functions for certain families of interesting rings and varieties. 

In the 1990s, a series of works were done for algebraic curves by Buchweitz, Chen, Han, Monsky, and Pardue (\cite{HaM, Par1994, BuC, Mon2, Mon3, Mon4}  in chronological order). They considered projective plane curves whose homogeneous coordinate rings are of the form of $k[x,y,z]/(g)$. 
Hypersurfaces in higher dimensions were considered by Han and Monsky~\cite{Ha1992, HaM}, Chang~\cite{ShTCh1993}, Chiang and Hung~\cite{ChH1998}. In addition, Seibert~\cite{Sei1997} worked on the Cohen-Macaulay rings of finite representation type.  
These studies revealed some unpredictable behaviors and mysterious nature of Hilbert-Kunz functions. 
Starting from the early 1990s, they have been investigated in waves of studies from different points of view with a variety of machinery. For instance, Han and Monsky developed the theory of representation rings that provided a means to compute these functions for diagonal hypersurfaces. Brenner, Frakhruddin and Trivedi studied the subject using sheaf theory (\cite{Br1, Br2, FaT}) and gave a systematic treatment for the case of smooth algebraic curves. Kurano linked algebraic intersection theory with Hilbert-Kunz function and produced an expression of the function in terms of local Chern characters (\cite{K16, K5}). Bruns and Watanabe offered a strong insight into Hilbert-Kunz function and multiplicity and proved that the computations in the setting of normal affine semigroup rings can be understood via Ehrhart theory (\cite{Bru1, Wat}). Nevertheless, it is in general difficult to compute Hilbert-Kunz function and its associated multiplicity. However, in light of Ehrhart theory and effective computer algebra systems developed in recent years, it is likely that we will be able to formulate accessible questions regarding Hilbert-Kunz functions. 

In comparison, many more studies have been done for Hilbert-Kunz multiplicity than its functional counterpart. 
The literature concerning Hilbert-Kunz multiplicity alone is a rich entity, and it is beyond the scope of this manuscript
to give a comprehensive account. The discussions in this paper are mainly focused on the functions.
While making no attempt to be complete, we mention some results on the multiplicity that are 
relevant to Hilbert-Kunz functions and link as many relevant scholarly works as possible. 

Valuable overviews of Hilbert-Kunz theory can be found in \cite{Br_arXiv} by Brenner that offers a very rich source  for the subject, and in \cite{Hu2013} by Huneke which provides alternative approaches to many results different from the original proofs. 

The outline of the paper is as follows. 
Section~\ref{open} introduces the definitions and a brief history of the theory. While presenting interesting questions arising from the literature, we also address how Hilbert-Kunz theory is related to other important notions and studies in commutative algebra. 
Section~\ref{techniques} reviews some of the techniques applied to the studies of Hilbert-Kunz functions. 
These include representation rings and $p$-fractals (\cite{Ha1992, HaM, Tei2002, MoT2004, MoT2006}), divisor class groups (\cite{HMM, ChK1}), the cohomology of vector bundles (\cite{Br1}), intersection theory (\cite{K16, K5}), and cellular decompositions on the fundamental domain (\cite{Bru1}). 
We pay close attention to the key steps and extract the crucial facts that contribute to establishing the targeted results. 
Readers may skip the subsection regarding each technique and return to it as needed. 
We are hopeful that our sketches may be helpful to those who wish to access the ideas in these proofs.
Section~\ref{affineEhrhart} is dedicated to normal affine semigroup rings. 
In this setting, Hilbert-Kunz function is closely related to the lattice point enumerator as shown first in Watanabe~\cite{Wat}. Bruns~\cite{Bru1} builds a pathway that bridges the Hilbert-Kunz and Ehrhart theories in a rigorous manner. 
This idea will be elaborated in this section. 
Section~\ref{example} includes examples done by counting lattice points or estimating with {\tt Macaulay2}. 
Most examples presented in Section~\ref{example} and comments about the techniques made 
throughout Sections~\ref{techniques}, \ref{affineEhrhart}, \ref{example}  are due to the author's own studies and observations. 
They can serve as starting points for further rigorous investigations.

\section{History in Brief}\label{open}

Throughout this paper, $p$ denotes a prime number and $e$ denotes a nonnegative integer 
which is often the argument of functions under consideration. 
Let $(R, \pp m)$ be a $d$-dimensional local ring of positive characteristic $p$.
For any ideal $I$, let $I^{[p^e]}$ denote the ideal 
generated by elements of the form of $a^{p^e}$ for all $a \in I$. 
Since $R$ has characteristic $p$, $I^{[p^e]}$ can in fact be generated by the 
$p^e$-th power of the elements in any given generating set of $I$. 
The ideal $I^{[p^e]}$ is called the $p^e$-th Frobenius power of  $I$ and
is an $\pp m$-primary ideal if $I$ is $\pp m$-primary. All modules in this paper are finitely generated. 

The Frobenius map $f: R \rightarrow R$ takes $r \in R$ to $f(r)=r^p$ and is a ring homomorphism.
Naturally one may consider $R$ as a module, denoted by $^1\!R$, over itself via restriction of scalars along $f$.
When $R$ is reduced, $f$ is injective so $R$ is isomorphic to its image $R^p$. In that case,
one can equivalently consider $R$ as a module over the subring $R^p$ which is how Kunz viewed it. 
Thus as an abelian group, $^1\!R$ is equal to $R$ but its module structure is  given by 
$r \cdot a = r^p a$ for all $r \in R$ and $a \in {^1\!R}$. 
We say $R$ is {\em $F$-finite} if the Frobenius map $f$ is a finite morphism. 
(See Miller~\cite{Mi2003} for a general background regarding the Frobenius endomorphism.)

In 1969, Kunz proved that the flatness of $R$ as an $R^p$-module characterizes its regularity. 
Precisely, Kunz~\cite{Kun} proved that $R$ is regular if and only if it is a reduced  flat  module over $R^p$. 
In fact his proof can be extended to show that the following are all equivalent without the reduced condition: 
(1) $R$ is regular; (2) the composed homomorphism $f^e$ is flat for all positive integers $e$; 
(3) $f^e$ is flat for some positive integer $e$ 
({\em c.f.} \cite[Theorem~21.2]{24hr}).
Kunz achieved this by studying the numbers
$\ell(R/m^{[p^e]})$ as $e $ increases and by applying Cohen's structure theorem for complete local rings. 
In particular, we have 

\begin{theorem}[Kunz~\cite{Kun}, 1969]\label{Kunz}
Let $R$ be a Noetherian local ring of dimension $d$. 
Then $R$ is regular if and only if $ \ell(R/m^{[p^e]}) = (p^e)^d $. 
\end{theorem}

Let $M$ be  a finitely generated $R$-module and $I$ an $\pp m$-primary ideal.
In 1983, Monsky~\cite{Mon1} called the function $\hk{M,I}{e} : \mathbb Z_{\geq 0} \rightarrow \mathbb Z$
\[ \hk{M,I}{e}: e \longrightarrow  \ell _A(M/I^{[p^e]}M) \]
the {\em Hilbert-Kunz function of $M$ with
  respect to $I$}.
Although it depends on
both $M$ and $I$, when there is no ambiguity, we will simply write it as $\hk{M}{e}$.
In particular, by the Hilbert-Kunz function of the ring $R$, we mean $M = R$ and $I=\pp m$, the maximal ideal.
For a finitely generated $R$-module $M$ of dimension $d$, Monsky also considered the limit 
\begin{equation}\label{HKmult}
 e_{HK}(M, I) := \lim_{e \rightarrow \infty} \frac{\hk{M,I}{e}}{(p^{e})^d} 
\end{equation}
and proved the following results :

\begin{theorem}[Monsky~\cite{Mon1}, 1983]\label{thmMonsky} 
Let $(R,\pp m)$ be a Noetherian local ring of positive characteristic
$p$ and $\dim R =d$. Let $I$ be an $\pp m$-primary ideal and $M$ a
finitely generated $R$-module of dimension $d$. Then

$(a)$ The limit in $($\ref{HKmult}$)$ exists and is always a positive real number.

$(b)$ The Hilbert-Kunz function is of the form of 
\begin{equation*}
\hk{M,I}{e} = e_{HK}(M,I)\, (p^e)^{d} + O ((p^e)^{d-1}).
\end{equation*}

$(c)$ If $R$ is a one-dimensional complete local domain and $\dim M=1$, then $\hk{M,I}{e} = e_{HK}(M,I) \cdot p^e + \delta_e $, 
where $e_{HK}(M;I)$ is an integer and $\delta_e$ is an eventually periodic function. 
\end{theorem}

In this paper, we write $g(x)=O(f(x))$ if there is a constant $C$ independent of $x$ such that 
$| g(x) | \leq C | f(x) |$ for $x$ large enough (or $x \gg 1$). 

Monsky named the limit in (\ref{HKmult}) the {\em Hilbert-Kunz multiplicity of
  $M$} {\em with respect to $I$}. The limit is exactly the coefficient of the dominating term $(p^e)^d$ in $\hk{M, I}{e}$. 
  If $I =\pp m$, we drop $I$ from the notation, and simply write $e_{HK}(M)$ and $\hk{M}{e}$. 
By replacing $R$ by $R/\rm{annih}(M)$, modulo the annihilator of $M$, we may always assume that $\dim M=\dim R$ and thus Theorem~\ref{thmMonsky} still holds.

Next we recall the familiar Hilbert-Samuel functions which concerns the length of $M$  modulo the usual powers of 
an $\pp m$-primary ideal $I$: 
\[ \HS_{M, I} (n) : = \ell ( M/I^n M). \]  
The Hilbert-Samuel multiplicity is defined as 
\[ i( M, I) : = d! \lim_{n \rightarrow \infty} \frac{ \HS_{M,I} (n) }{n^d} . \]
There exists a polynomial $\HSP_{M,I}(n)$ with rational coefficients so that 
$ \HS_{M,I} (n) = \HSP_{M,I} (n)$ for all large enough $n$. More precisely,
\[ \HSP_{M,I} (n) = \frac{1}{d!} i(M,I) \, n^d + \text{ (lower degree terms in $n$).}\] 

Also commonly known and studied are the Hilbert functions for graded rings or modules, and Buchsbaum-Rim function
which is the module version of Hilbert-Samuel function via Rees rings
(c.f. \cite{BrH,CLU,R1}).
Similar to Hilbert-Samuel function, Hilbert function and Buchsbaum-Rim function are polynomials of $n$ for large $n$, and their leading coefficients are always rational numbers.  

However, unlike the typical Hilbert-type functions just mentioned, the behavior of Hilbert-Kunz functions is rather unpredictable 
as seen in the next two examples. 
In the next section, upon the consideration of Theorem~\ref{RobFrob}, we will provide some observations that set apart Hilbert-Kunz functions from 
Hilbert and Hilbert-Samuel functions. 
These obersvaions also show that it is often necessary to consider the higher homology modules when it comes to Hilbert-Kunz functions.  

\begin{example}[Kunz~\cite{Kun}, 1969, Example~4.6(b)]\label{exKunz}
Set $R = \kappa [[ x,y]]/ (y^4 - x^3 y)$ (corresponding to four different lines) where $\kappa$ is an algebraically closed field. 
Then $\hk{R}{e} = 4 p^e - 3  , \text{ if } p \equiv 1 \pmod 3 $; and 
$\hk{R}e = 4 p^e  + \delta_e $, if $p \equiv -1 \pmod 3$ where $\delta_e= -3 $ if $e$ is even and $-4$ if $e$ is odd. 
\end{example} 

\begin{example}[Monsky~\cite{Mon1}, 1983, p.~46]\label{exMonsky} 
Set $R = \mathbb Z/p [[ x,y]]/(x^5 - y^5)$ and $p \equiv \pm 2 \pmod 5$. Then
 $\hk {R}e = 5 p^e + \delta_e$ where $\delta_e = -4$ if $e$ is even and $-6$ if $e$ is odd.
\end{example}

There is no effective algorithm for computing Hilbert-Kunz multiplicity. 
Hilbert-Kunz multiplicity was long thought to be rational, but this was proved not to be the case (see \cite{Br_arXiv}). 
In general, the question regarding the rationality of Hilbert-Kunz multiplicity has attracted serious research efforts since the notion 
was introduced in \cite{Mon1} and the debate was once a popular pastime for decades.  
The question regarding how to effectively compute Hilbert-Kunz function remains widely open. 

Hilbert-Kunz function was initially defined for local rings. 
One may apply the same definition for graded rings over a field $\kappa$ of positive characteristic 
with respect to their graded maximal ideals, or 
affine semigroup rings over such $\kappa$ with respect to the maximal ideals generated by elements in the semigroup. 
Which case we are dealing with will be clear from the context in each section or example.

\subsection{Multiplicity}\label{sub_multi} 

Although this article focuses on Hilbert-Kunz functions, it would be incomplete without touching upon some developments on the multiplicity. 

By Theorem~\ref{Kunz}, being regular for a local ring $R$ is equivalent to having a nice Hilbert-Kunz function, which in turn implies 
$e_{HK}(R) =1$.  Then Kunz asked if the property $e_{HK}(R,\pp m)=1$ is sufficient for $R$ being regular. This was proved 
by Watanabe and Yoshida~\cite{WY1} for unmixed local rings, and also by Huneke and Yao~\cite{HuY} without 
using the tight closure techniques as done in \cite{WY1}. 

A series of studies regarding  various levels of singularities and the estimation of lower
bounds for Hilbert-Kunz multiplicities can be found in the works of Watanabe, Yoshida, Blickle, Enescu, Shimomoto, Aberbach, 
Celikbas, Dao, Huneke, and Zhang
\cite{WY2, WY4,  WY3, BlE, EnS, AbEn2008, AbEn2013, CeDHZ2012}.

Another algebraic notion related to  singularities is tight closure, introduced by Hochster and Huneke~(\cite{Ho2004,HoH1990,Hu1998}). 
The subjects of tight closure and Hilbert-Kunz multiplicities share a close relationship. 
In fact, under some mild conditions, the tight closure of an $\pp m$-primary ideal $I$ is the largest ideal containing $I$ that has the same Hilbert-Kunz multiplicity as $I$~(\cite{HoH1990}). 
The introduction of \cite{EtY} by Eto and Yoshida offers a nice comparison between Hilbert-Kunz and Hilbert-Samuel multiplicities. It describes a parallel relationship of integral closure to Hilbert-Samuel multiplicity versus  tight closure to Hilbert-Kunz multiplicity. The literature on the relationship between these two notions is extremely rich. 
Before leaving this topic, we point out yet another connection between the two theories. By understanding the semistability of vector bundles on projective curves, Brenner brought out the geometric interpretation of tight closures and successfully solved some open problems in the tight closure theory~({\em c.f.} \cite{Br3lecture}). This technique has been proved to be effective also in Hilbert-Kunz theory which will be reviewed in Subsection~\ref{sub_sheaf}. 

As we will see in the next section, it is sometimes natural to consider the Euler characteristic after applying 
the Frobenius functor, and the Hilbert-Kunz function can be expressed in terms of the Euler characteristic. 
By the singular Riemann-Roch theorem, Kurano made a connection between the Hilbert-Kunz function with local Chern characters and proved that the Hilbert-Kunz multiplicity characterizes rings that are numerically equivalent to 
Roberts rings. (See Subsection~\ref{sub_Chern} for {\em numerical equivalence}.)  
In the 1960s, Serre defined the intersection multiplicity in terms of the $\tor$-functors~(\cite{Ser1961}). 
Several conjectures followed, some of which remain unsolved.
Roberts proved the vanishing conjecture for complete intersections and isolated singularities by using intersection theory~(\cite{Ro1}). 
Independently, this was also proved by Gillet and Soul\'e using $K$-theory~(\cite{GiSo1985, GiSo1987}). 
Kurano called a ring $R$ a {\em Roberts ring} 
if the only nontrivial component in the Todd class of $R$ belongs to the subgroup of codimension zero in the Chow group~(\cite{Ku2001}). 
This condition is satisfied by rings for which the vanishing theorem holds in Roberts' result. 
If $R$ is further assumed to be a homomorphic image of a regular local ring and also Cohen-Macaulay, then it is 
numerically equivalent to a Roberts ring if and only if the Hilbert-Kunz multiplicity satisfies the condition $e_{HK}(I) = \ell (R/I)$ for any $\pp m$-primary ideal $I$ of finite projective dimension~(\cite{K16}). 

Going back to the mystery of the multiplicity, Monsky initially suspected that $e_{HK}(R)$ 
should always be rational~(\cite[1983]{Mon1}).  
Even though all the examples with explicitly known Hilbert-Kunz multiplicities do take on rational values,  
the question itself remained stubbornly unresolved. Later, Monsky conjectured otherwise in \cite[2008]{Mon6}. Eventually Brenner~\cite[2013]{Br_arXiv} constructed a module that has an irrational multiplicity leading to the existence of three dimensional rings with irrational Hilbert-Kunz multiplicity.

Still it is natural to ask under what conditions or for what family of rings the Hilbert-Kunz multiplicities are rational. 
Some of the known cases include 

\begin{enumerate} 
\item Regular local rings ($e_{HK}(R) =1$, Kunz\cite{Kun}).
\item Complete local domains of dimension one ($e_{HK}(R) \in \mathbb Z$, Monsky~\cite{Mon1}). 
\item Algebraic curves -- two-dimensional standard graded algebra over an algebraically closed field: 
\begin{itemize} 
\item $R$ is normal (Brenner~\cite{Br2} and Trivedi~\cite{Tr2005} independently). 
\item $R = k[x,y,z]/ (f)$ for plane cubics ($\deg f =3$) (Monsky~\cite{Mon2, Mon5, Mon2007, Mon2011}, Buchweitz and Chen~\cite{BuC}, Pardue~\cite{Par1994}).
\end{itemize} 
\item Hypersurfaces 
\begin{itemize}
\item Diagonal hypersurfaces (Han and Monsky\cite{Ha1992, HaM}) 
\item Some special families (Kunz~\cite{Kun1976}; Monsky and Teixeira \cite{Tei2002, MoT2006}) 
\end{itemize}
\item $F$-finite Cohen-Macaulay rings of finite Cohen-Macaulay type (Seibert~\cite{Sei1997}). 
\item Full flag varieties and elliptic curves (Fakhruddin and Trivedi~\cite{FaT}). 

\item Stanley-Reisner rings  and binomial hypersurfaces (Conca~\cite{Co1996}). 

\item Segre product of polynomial rings (Eto and Yoshida~\cite{EtY}). 

\item Segre product of any finite number of projective curves (Trivedi~\cite{Tr2018}). 

\item Normal affine semigroup rings (Watanabe~\cite{Wat}). 
\end{enumerate} 

In \cite{Wat}, the value $\ell(R/I^{[p^e]})$ is obtained by counting the lattice points within a certain relevant region defined by the affine semigroup. 
This approach inspires our discussions in Sections~\ref{affineEhrhart} and \ref{example} for the Hilbert-Kunz functions. 

It would be interesting to know what other families, especially in higher dimensions, produce rational $e_{HK}(R)$. 
More generally, what are the hidden conditions or properties common to these examples? 

\subsection{Functions}\label{sub_fcn}

From now on, we will use $q$ to denote $p^e$ unless we want to emphasize its explicit dependence on $e$.  
If $\hk{}{e}$ is expressed as a function in $p^e$, we write $\hkq{}{q}$ to highlight this feature, namely, $\hk{}{e} = \hkq{}{q}$. 
We will also use $\alpha$ in place of $e_{HK}$. Hilbert-Kunz functions have the following functional form for $q \gg 1$:  

By Theorem~\ref{Kunz}  (Kunz~\cite[1969]{Kun}), 
if $R$ is a regular local ring of dimension $d$, then 
\begin{equation}\label{fcn:Kunz}
 \hk{R,\pp m}{e} = \hk{R}{e} = \hkq{R}{q} = q^d .
 \end{equation}
 
By Theorem~\ref{thmMonsky}($b$) (Monsky~\cite[1983]{Mon1}), for  an arbitrary Noetherian local ring $R$ and 
finitely generated $R$-module $M$ with respect to ideal $I$, 
\begin{equation}\label{fcn:Monsky}
\hk{M, I}{e} = \hkq{M,I}{q} = \alpha \, q^{d} + \mathcal O(q^{d-1}).
 \end{equation}

Huneke, McDermott and Monsky~\cite[2003]{HMM} refines (\ref{fcn:Monsky}) for an excellent normal domain $R$ with perfect residue field to the following form which identifies the next order term
\begin{equation}\label{fcn:HMM}
\hk{M,I}{e}=\hkq{M,I}{q}= \alpha \, q^{d} + \beta \, q^{d-1}+  \mathcal O(q^{d-2}),
\end{equation}
where $\beta$ is called the {\em second coefficient}. It is written as $\beta_I(M)$  
if we wish to emphasize that $\beta$ is a function of $M$ with respect to $I$ or $\beta(M)$ if $I$ is the maximal ideal. 
The normality condition on the ring was further relaxed by Kurano and the author~\cite{ChK1}, and independently Hochster and Yao~\cite{HoY}.

As noted in \cite{HMM}, the third coefficient in (\ref{fcn:HMM}) in general does not exist.
For example, as computed by Han and Monsky~\cite{HaM}, the domain
$R = \mathbb Z/5\mathbb Z[ x_1, x_2, x_3, x_4] / (x_1^4 + x_2^4 + x_3^4 + x_4^4) $
has  $\hk{R}{e} = \frac{168}{61} (5^n)^3 - \frac{107}{61} (3^n)$. 
The end term $- \frac{107}{61} (3^n)$ is $O( 5^n)$ as expected in (\ref{fcn:HMM}) but is not in the form of  $c( 5^n )+ O(1)$ with some constant $c$. 
In Subsection~\ref{sub_classgp}, we will point out that the normality  condition (R1)$+$(S2) can be loosened to be (R1$'$), 
a condition similar to (R1). However, as seen in Examples~\ref{exKunz} and \ref{exMonsky}, this new condition cannot be further relaxed. 

In the earlier studies of Hilbert-Kunz theory, the entire function was computed in many cases. 
These include certain projective plane curves and hypersurfaces by Kunz, Han, Monsky, Pardue, Buchweitz, and Chen ~\cite{Kun1976, HaM, Par1994, Co1996, BuC, Mon2, Mon3, Mon4}, Stanley-Reisner rings by Conca~\cite{Co1996}, and Cohen-Macaulay rings of finite representation types by Seibert~\cite{Sei1997}. The outcome of these investigations shows that Hilbert-Kunz functions are periodic in some cases but not always. 

Despite this evidence, it is still interesting to ask how close the Hilbert-Kunz functions are to the form of a polynomial in $q$, or the next best possibility,  a quasipolynomial in $q$. By a quasipolynomial, we mean a function in the form of a polynomial whose coefficients are periodic functions. The functions obtained in Examples~\ref{exKunz} and \ref{exMonsky} are quasipolynomials. 
Some cases have been studied and computed, for instance Conca~\cite{Co1996}, Seibert~\cite{Sei}, Miller, Robinson and Swanson~\cite{MiS2013, RoSw2015}, but a systematic treatment is still lacking. 
This will be discussed in Sections~\ref{affineEhrhart} and \ref{example} for normal affine semigroup rings.

Before moving on to the next section for some established techniques applied in the study of Hilbert-Kunz functions, we briefly mention some more recent developments related to Hilbert-Kunz theory.  

In recent years, Hilbert-Kunz multiplicity has been generalized to be taken with respect to ideals that are not primary to the maximal ideal. This is done by using the length of the local cohomology at the maximal ideal. This notion, known as the {\em generalized} Hilbert-Kunz multiplicity, was proved to exist and further developed by Epstein and Yao~\cite{EpY2017}, Hern\'andez and Jeffries~\cite{HeJ2018}, and Dao and Smirnov~\cite{DaS2020}. It should be noted that this is different from another generalized version for monomial ideals, primary to a maximal ideal, considered by Conca, Miller, Robinson and Swanson~\cite{Co1996, MiS2013, RoSw2015} (see Subsection~\ref{sub_comb}). 

A potential extension of Hilbert-Kunz multiplicity to the characteristic zero setting has been proposed. 
In this regard, {\em limit} Hilbert-Kunz multiplicity is considered. When the limit exists, its uniformity is also of interest; see Application~\ref{limitHK} for  more details and relevant references. 

It is also known, from the work of 
Huneke and Leuschke~\cite{HuL2002}, Singh~\cite{Si2005}, Yao~\cite{Yao2006},  and Aberbach~\cite{Ab2008},
that the theory of $F$-signature is closely related to that of Hilbert-Kunz multiplicity . 
In particular, it is proved by Tucker, Polstra, Caminata and De Stefani~\cite{Tu2012, PoT2018, CaDeS2019} 
that $F$-signature functions have the same approximate functional forms as Hilbert-Kunz functions. 
In the case of affine semigroup rings, the work of Watanabe~\cite{Wat}, Bruns~\cite{Bru1}, Singh~\cite{Si2005}, 
and Von Korff~\cite{VonK2011} show that the same effective methods can be used to compute
these two sets of functions and multiplicities. We will present the techniques of \cite{Bru1} in Subsection~\ref{sub_BG}. 


\section{Techniques in Hilbert-Kunz Functions}\label{techniques}  

In this section, we review some of the techniques that have been applied repeatedly in investigating Hilbert-Kunz functions. 
In doing so, we wish 
to address the depth of Hilbert-Kunz theory, and to make the technical proofs more accessible.
The review here should be taken as complementary reading to the original proofs.
Our summaries are not meant to replace them.

To keep the paper self-contained, we will define the terms commonly used in this paper to avoid confusion and ambiguity. 
However, for the definitions of those notions that require separate background, 
we will refer to appropriate references so that the discussions remain focused on the key ideas.
In these cases, we provide a means of processing them to help the readers move forward. 

Unexplained notation and terminology can be looked up in the following references: 
Cutkosky~\cite{Cut}, Hartshorne~\cite{Ha} and Shafarevich~\cite{Sh1,Sh2} for algebraic geometry; 
Fulton~\cite{F} and Roberts~\cite{R1} for intersection theory.
For an algebraic approach to locally free sheaves as modules and Chern characters, we recommend Part II, especially Chapter 9, of \cite{R1}. 
An affine semigroup ring can sometimes be viewed as a subalgebra generated by finitely many monomials of a polynomial ring over a field, but we highly recommend some basic knowledge in their connection to toric varieties as presented in Cox, Little and Schenck~\cite{CLS}. Stanley~\cite{St2} and Miller and Sturmfels~\cite{MSt} are excellent sources for combinatorics theory. 

Whenever possible, we adopt the notation  from the original papers from which we are presenting the techniques. 
There may be some overlapping notation. But we believe the confusion is minimal as they appear in different contexts.

We begin by presenting two functors arising from the Frobenius map. 

Let $R$ be a Noetherian local ring of positive characteristic $p$ with perfect residue field.
Recall that the Frobenius homomorphism is the map $f: R \rightarrow R$, $f(r)= r^p$. The Frobenius homomorphism induces two functors. One is the {\em functor by restriction of scalars}: 
\[M \mapsto {^1\!M} \]
where $^1\!M$ is the module over $R$ via $f$. 
Precisely, for any $m \in M$, if $(m)_1$ denotes the element $m$ considered as an element in $^1\!M$, then
$r\cdot (m)_1 = (r^{p} m)_1$ for any $r \in R$.
The functor by the restriction of scalars $M \mapsto {^1\!M}$ is an exact functor ({\em c.f.} \cite[Section~7.3]{R1}). 
In what follows, we assume that $R$ has a perfect residue field and that $^1\!R$ is a finitely generated $R$-module, that is, $R$ is $F$-finite. 
Note that the ring $R$ is $F$-finite if $R$ is a complete local ring with perfect residue field or a localization of a ring of finite type over 
a perfect field. 
If $R$ is $F$-finite, the rank of ${^1\!R}$ is $p^d$ if rank is defined ({e.g.} $R$ is a domain). 
If $R$ is regular, then by Kunz's theorem, ${^1\!R}$ is flat so it is also free.
For any positive integer $e$, the module $^e\!M$ is obtained by the self-composition of the functor;
or equivalently,  $^e\!M$ is the module over $R$ via the composite function $f^e=f \circ \cdots \circ f$. 

The other functor is the {\em Frobenius functor} $\Ff$ (also known as the {\em Peskine-Szpiro functor}) from the category of left $R$-modules to itself via tensor product over $f$, 
\[ \Ff(M) : = {^1\!R} \otimes _R M. \] 
This is considered as an extension of scalars along $f$. The resulting $R$-module structure on $\Ff(M)$ is via the left factor. 
More precisely, for any $a, r \in R$ and $m \in M$, $a (r \otimes m) = ar \otimes m$ and $1 \otimes rm = r^p \otimes m$. 
As left $R$-modules, $\Ff(R) \cong R$. 
Therefore, 
\[ \Ff(R/I) = {^1 R} \otimes_R R/I = {R}/ f(I) = R/I^{[p]}\]
with its ordinary left scalar multiplication. One can iterate this process and see that $\Ff^e(R) = R$ and $\Ff^e(R/I) = R/I^{[p^e] } $.

\begin{remark}\label{rmk_Frobenius}
Due to commutativity, $^1\!R \otimes_R M \cong M\otimes_R {^1\!R}$, but the module structure of 
$M\otimes_R {^1\!R}$ is via the right factor. That is, for any $a, r \in R$ and $m \in M$, 
$a( m \otimes r) = m \otimes ra$ and $m r \otimes 1 = m \otimes r^p$. 
\end{remark}

The Peskine-Szipiro functor is covariant, additive, and right exact from the category of left $R$-modules to itself. 
When $\Ff$ is applied to a complex of free modules of finite ranks, all free modules remain the same up to isomorphism since tensor product commutes with direct sum, but the entries of the matrices of the boundary maps are raised to the $p$-th power. A notable theorem by Peskine and Szpiro states that if $\mathbb G_{\bullet}$ is a finite complex of finitely generated free $R$-modules, then $\mathbb G_{\bullet}$ is acyclic if and only if $\Ff(\mathbb G_{\bullet})$ is acyclic ({\em c.f.} \cite[Theorem~8.7]{BrH}).
 
Next we present an analogous Hilbert function for Frobenius powers. 
Its form is proved by induction on dimension similarly to the case of usual Hilbert functions, but the induction steps are more complicated~(\cite{Sei, R1}). 
Let $\mathcal C$ be a category of finitely generated $R$-modules of dimension no bigger than a nonnegative integer $n$. 
We say that  $\mathcal C$ satisfies the property ($\dagger$) if whenever $M', M, M''$ in $\mathcal C$ form a short exact sequence 
$0 \rightarrow M' \rightarrow M \rightarrow M'' \rightarrow 0$, then $M$ is in $\mathcal C$ if and only if $M'$ and $M''$ are in $\mathcal C$. 
The property ($\dagger$) implies that $M \in \mathcal C$ if and only if $R/\pp p \in \mathcal C$ for all prime ideals in the support of $M$. 
Since $M$ and $^1M$ have the same support, it follows that if $M\in \mathcal C$, then $^1\!M \in \mathcal C$. 

\begin{theorem}[Seibert~\cite{Sei}; Roberts~\cite{R1}, Theorems~7.3.2 \& 2.3.3]\label{RobFrob} 
Let $\mathcal C$ be a category of finitely generated $R$-modules of dimension no bigger than a nonnegative integer $n$. 
Assume that $\mathcal C$ satisfies the property $(\dagger)$. 
Let $g$ be a function from $\mathcal C$ to $\mathbb Z$. 

$(a)$ If $g$ is additive on short exact sequences, then for any $M \in \mathcal C$, 
there exists a polynomial in $p^e$ with rational coefficients $a_0, \dots, a_n$ such that 
\[ g( ^eM ) = a_n (p^e)^n + a_{n -1} (p^e)^{n -1} + \cdots + a_0\]
for all integer $e \geq 0$. 

$(b)$ If $g(M) \leq g(M') + g(M'')$ whenever $0 \rightarrow M' \rightarrow M \rightarrow M'' \rightarrow 0$ is exact in 
$\mathcal C$, then there exists a real number $c(M)$ such that 
\[ g( {^eM}) = c(M) (p^e)^n + O( (p^e)^{n-1})\]
for all integers $e \geq 0$.
\end{theorem}

In Seibert~\cite{Sei} and Roberts~\cite[Section~3]{R1}, one can find interesting applications of Theorem~\ref{RobFrob} 
to the existence of Hilbert-Kunz multiplicity originally proved by Monsky~\cite{Mon1} and the asymptotic Euler characteristic $\chi_{\infty}$ defined by Dutta~\cite{Du1}. The latter is now known as {\em Dutta multiplicity}. 

Now we provide some elementary observations based on Theorem~\ref{RobFrob}$(a)$. 
Let $I$ be an $\pp m$-primary ideal with a finite free resolution. 
Then $R/I$ has a finite free resolution $\mathbb G_{\bullet}$, 
{i.e.,} $\mathbb G_{\bullet} \rightarrow R/I \rightarrow 0$ is exact.
Now let $\mathcal C$ be the category of finitely generated $R$-modules. 
The {\em Euler characteristic} $\chi_{\mathbb G_{\bullet}}$ of a module $M \in \mathcal C$ is defined to be 
the following alternating sum of lengths of the homology modules
\[ \chi_{\mathbb G_{\bullet}}(M) : = \sum_i (-1)^i \ell(  H_i( \mathbb G_{\bullet} \otimes_R M ) . \]
The Euler characteristic is additive on short exact sequences. 
Since restriction of scalars along $f$ is exact, therefore, by Theorem~\ref{RobFrob}$(a)$,
$ \chi_{\mathbb G_{\bullet}}( ^e M) ) $ 
is a polynomial in $p^e$ with rational coefficients.

We observe that 
\[ \mathbb G_{\bullet} \otimes_R {^e\!M}  \cong  \mathbb G_{\bullet} \otimes_R {^e\!(R \otimes_R M )} \cong  
( \mathbb G_{\bullet} \otimes_R {^e\!R}) \otimes_R M  \cong \Ff^e(  \mathbb G_{\bullet}) \otimes_R M . \]
The last equality holds for 
$\mathbb G_{\bullet} \otimes_R  {^e\! R}  \cong {^e\! R} \otimes_R \mathbb G_{\bullet} = \Ff^e(  \mathbb G_{\bullet})   $
by Remark~\ref{rmk_Frobenius}.
Upon taking $M=R$, we have that

\begin{equation}\label{euler} 
\begin{array}{lll}
\chi_{\mathbb G_{\bullet}}( ^eR ) 
 & = &  \sum_i (-1)^i \ell(  H_i( \Ff^e (\mathbb G_{\bullet} ) ) )  \\
 & = &   \ell(  H_0( \Ff ^e(\mathbb G_{\bullet} ) ) ) \\
 & = & \ell (  R/I^{[p^n]} ) = \hk{R,I}{e}.
 \end{array} \end{equation} 
In the display above, the second equality holds since $\Ff$ preserves acyclicity of finite free complexes due to 
 the result of Peskine and Szpiro mentioned above. Hence the higher homology modules all vanish. 
 The third equality is due to the fact that $\Ff$ is right exact so $H_0( \Ff ^e(\mathbb G_{\bullet} ) =  \Ff^e (R/I)$. 
 Applying Theorem~\ref{RobFrob}$(a)$ to $\chi_{\mathbb G_{\bullet}}$ in (\ref{euler}), we obtain the 
following corollary for the Hilbert-Kunz function.  

\begin{corollary}\label{regular}
Let $(R, \pp m)$ be a local ring of characteristic $p$ with perfect residue field and
$I$ an $\pp m$-primary ideal.  Assume that $I$ has finite projective dimension. 
Then $\hk{R,I}{e}$ is a polynomial in $p^e$ with rational coefficients. 
\end{corollary} 

For $\hk{M,I}{e}$, if $M$ is a finitely generated flat $R$-module, then by a parallel argument as in (\ref{euler}), we have 
$ \chi_{\mathbb G_{\bullet}}( ^e M ) = \hk{M,I}{e}$. But a finitely generated flat module over a local ring is free. So $\hk{M,I}{e}$ is just a multiple of $\hk{R,I}{e}$. 

Corollary~\ref{regular} imposes a strong condition: $I$ has finite projective dimension.
Commonly interesting ideals in any given ring $R$ do not necessarily have finite projective dimension unless $R$ is regular. 
If $R$ is regular, then by Corollary~\ref{regular}, 
we have $\hk{R,I}{e} = \chi_{\mathbb G_{\bullet}}( \Ff^e(R) ) 
= a_d (p^e)^d + a_{d-1} (p^e)^{d-1} + \cdots + a_0$ for any $\pp m$-primary ideal $I$. As with classical Hilbert polynomials, 
an immediate challenge is how one can identify the coefficients. Suppose we can use anything at our disposal, 
then we can jump directly to Theorem~\ref{fcn:Kurano} which shows that 
each coefficient $a_i$ is determined by a certain class in the Chow group $A_i(R)$. But the Chow group of a regular ring has only the top piece, namely $A_i(R) =0$ for $i = d-1, \dots, 0$. Hence $\hk{R,I}{e} = a_d (p^e)^d$ and indeed $a_d = \ell(R/I)$ by \cite[Theorem~6.4]{K16}, recovering Kunz's original theorem when $I= \pp m$. 
This is a long detour back to where we started. 
We note that even if Hilbert-Kunz function is eventually a polynomial in $p^e$, 
determining the values of the coefficients is a very difficult task. 

In general for an arbitrary finitely generated module $M$, the higher homology modules in
\begin{equation}\label{eqFrob}
 \chi_{\mathbb G_{\bullet}}( ^e M ) =   \sum_i (-1)^i \ell(  H_i( \Ff^e (\mathbb G_{\bullet} ) ) \otimes_R M  )
\end{equation}
do not always vanish. 
In such a case,   
one has to  study and control the higher homology modules in order to properly describe 
the desired length of the zeroth homology that equals $\hk{M,I}{e}$. 
Subsections~\ref{sub_classgp} and \ref{sub_sheaf} will demonstrate how this can be done in their respective cases. 

For an estimate of the length of higher homology modules in a latter discussion, we quote a theorem from \cite{HMM}, which can be traced back to \cite{R7, HH, Sei} and later strengthened by \cite{ShTCh1997}.
It states that if $\mathbb G_{\bullet}: 0 \rightarrow G_n \rightarrow \cdots \rightarrow G_0 \rightarrow 0$ is a finite complex of free modules and each homology module 
$H_i(\mathbb G_{\bullet})$ has finite length. Let $d=\dim M$. Then for $e \geq 0$,
\begin{equation}\label{boundlength}
\ell( H_{n-t}( \Ff^e ( \mathbb G_{\bullet} ) \otimes_R M) ) = O( (p^e)^{min(d, t) } ) .
\end{equation}

In what follows, 
we will explore how each path leads us to approximate or compute Hilbert-Kunz functions.

\subsection{Via Representation Rings and $p$-Fractals}\label{Monsky} 
In this subsection, we describe representation rings and their connections to Hilbert-Kunz theory. 
Then we briefly describe the next set of activity that it led to, namely the highly technical theory of $p$-fractals.  
At the end we mention a conjecture about the behavior of Hilbert-Kunz multiplicity when the characteristic $p$ increases. 

Pioneered by Example~4.3 in Kunz~\cite{Kun1976} and Monsky's initial paper \cite{Mon1}, a series of studies on the Hilbert-Kunz function of hypersurfaces followed in the 1990s. These include works by Buchweitz, Chang, Chen, Han, Monsky and Pardue~\cite{HaM, Par1994, BuC, Mon2, Mon3, Mon4} and for Cohen-Macaulay rings of finite representation type by Seibert~\cite{Sei1997}.
Especially remarkable is the paper 
``Some surprising Hilbert-Kunz functions'' by Han and Monsky~\cite{HaM}.
It introduced the idea of using representation rings in the setting of the affine coordinate rings of the diagonal hypersurfaces,
$\kappa[x_1, \dots, x_s]/(\sum_i x_i^{d_i})$,
which became the first systematic method for computing Hilbert-Kunz functions. It also identified these rings
as the first known family whose Hilbert-Kunz functions are eventually periodic.
For the detailed structure of this special representation ring, 
we refer to \cite{Ha1992, HaM, Tei2002} but we describe the basics here.  
Out of this grew the method of $p$-fractals of Monsky and Teixeira for handling other hypersurfaces, 
which we mention subsequently; for further details see \cite{MoT2004, Tei2002}.

{\bf \em Representation Rings.} Here we describe the relevant representation rings and 
how they apply to computing Hilbert-Kunz functions. 
Let $\kappa[T]$ be a polynomial ring in one variable and
$\mathscr C$ the category of finitely generated modules over $\kappa[T]$ annihilated by a power of $T$. 
Then $\mathscr C$ satisfies the property ($\dagger$) defined previously, namely, for any short exact sequence 
\[0 \rightarrow M' \rightarrow M \rightarrow M'' \rightarrow 0, \] 
$M$ is in $\mathscr C$ if and only if $M'$ and $M''$ are in $\mathscr C$. 
We further define the Grothendieck group  $K_0(\mathscr C)$ of $\mathscr C$ as the free abelian group generated by 
the isomorphism classes $[M]$ of objects $M$ in $\mathcal C$ modulo the subgroup generated by $[M]-[M']-[M'']$ for 
any $M, M', M''$ satisfying the above short exact sequence. 

Let $[M]$ and $[N]$ represent two classes in $K_0(\mathscr C)$. Consider the tensor product of $M$ and $N$ over $\kappa$. We define an action of $T$ on $M\otimes_{\kappa} N$ by 
$T(m\otimes n) = T(m) \otimes n + m \otimes T(n)$ for any $m\in M$ and $n\in N$. 
Thus $M\otimes_{\kappa} N$ is in $\mathscr C$. It can be verified that  
$K_0(\mathscr C)$ is a commutative ring with the binary operations $+$ and $\cdot$ induced by 
 $\oplus$ and $\otimes$ respectively, and that the classes of the zero module and $\kappa[T]/(T)$ are 
the respective additive and multiplicative identities (\cite[Theorem~1.5]{HaM}). 
The Grothendieck group $K_0(\mathscr C)$ with this commutative ring structure is called the {\em representation ring} in \cite{HaM}, where it is denoted by $\Gamma$. 

Since every module $M$  in  $\mathscr C$ is finitely generated and annihilated by a power of $T$, 
by the structure theorem for modules over a P.I.D., $M$ can be decomposed uniquely up to isomorphism
into a finite direct sum
as $M \cong \oplus_j \kappa[T]/(T^{n_j})$. 
Let $\delta _j$ denote the class represented by $\kappa[T]/(T^{j})$ in $\Gamma$. Then 
$\{ \delta_j\}_{j=1}^{\infty}$ forms a basis for $\Gamma$ as an abelian group. 
Moreover, one can define  a map  $\alpha :  \Gamma  \rightarrow  \mathbb Z$   
that takes $[M]$ to the number of indecomposable summands in the unique decomposition just mentioned.
The definition of $\alpha$ here is equivalent to that in \cite{HaM} where for technical reasons 
$\alpha$ is defined via a different basis $\lambda_j = (-1)^j(\delta_{j+1} - \delta_j)$. 

We observe that if $V=\kappa[T]/(T^j)$, then $\dim_{\kappa}( V/TV ) =1$. Since 
$M \cong \oplus_j \kappa[T]/(T^{n_j})$ is a unique decomposition, and 
$\alpha( [M])$ equals the minimal number of generators of $M$ as a $\kappa[T]$-module, we have 
 $\alpha( [M]) = \dim_{\kappa}(M/TM)$.  

Now we may present the connection between the representation ring and the Hilbert-Kunz function of 
a diagonal hypersurface, {i.e.,} a ring of the form 
\[ R = \kappa[x_1, \dots, x_s]/(x_1^{d_1} + \cdots + x_s^{d_s}). \]  
Without loss of generality, we assume $d_i >0$. 
First we take $q$ as an integer larger than all the $d_i$'s and view $M_i= k[x_i]/(x_i^{q})$ as a $\kappa[T]$-module with $T$ acting as multiplication by $x_i^{d_i}$. 
As a $\kappa$-vector space, 
\[ M_i = \kappa \cdot 1 + \kappa \cdot x_i + \kappa  \cdot x_i^2 + \cdots + \kappa x_i^{d_i-1} + \kappa \cdot x_i^{d_i} + 
 \kappa  \cdot x_i^{d_i+1} + \cdots + \kappa  \cdot x_i^{q-1} .\]
Then, there are invariant  $\kappa[T]$-submodules generated by 
$1, x_i, x_i^2, \dots, x_i^{d_i-1}$ 
respectively. Write $q = k_i d_i + a_i$ with $0 \leq a_i < d_i$. The invariant submodule generated by $x_i^c$, 
with $c=0, \dots, a_i-1$ are annihilated by $T^{k_i+1}$ but no smaller power. 
This is true because 
$T^{k_i} \cdot x_i^c = x_i^{k_i d_i} \cdot x_i^{c} \neq 0$ in $M_i$ if $c <a_i$, but
$T^{k_i+1} \cdot x_i^c = x_i^{k_i d_i + a_i + (d_i -a_i) +c} = x_i^q \cdot x_i^{d-a_i + c} =0 $  in $M_i$ and 
$0< d_i -a_i \leq d_i - a_i + c \leq d_i -1$ if $0 \leq c \leq a_i-1$. 
Similarly, the invariant submodules generated by $x_i^c$, $c=a_i, \dots, d_i-1$, are annihilated by $T^{k_i}$ but no smaller power. Hence as a $\kappa[T]$-module, $M_i$ can be decomposed into 
\[ M_i \cong \left (\kappa[T]/(T^{k_i+1}) \right)^{a_i} \oplus \left ( \kappa[T]/(T^{k_i}) \right)^{d_i-a_i} . \]
Hence we can conclude that $[M_i] =  a_i \delta _{k_i+1} + (d_i-a_i) \delta _{k_i}$ in $\Gamma$. 

Next we assume that $\kappa$ is a field of positive characteristic $p$ and recall that $q=p^e$. Observe that 
\[ M_1\otimes_{\kappa} \cdots \otimes_{\kappa} M_s \cong k[x_1, \cdots, x_s]/ (x_1^q, \dots, x_s^q) 
 = k[x_1, \cdots, x_s]/ \pp m^{[q]} , \]
 where $\pp m $ is the ideal $(x_1, \dots, x_s)$. By the definition of the action $T$ on each $M_i$ and the multiplication structure in $\Gamma$, $M_1\otimes_{\kappa} \cdots \otimes_{\kappa} M_s$ is a 
 $\kappa[T]$-module where $T$ acts by multiplication by $x_1^{d_1} + \cdots + x_s^{d_s}$. 
 On the other hand, if we let $M$ denote $M_1\otimes_{\kappa} \cdots \otimes_{\kappa} M_s$, 
 then $M/TM$ is obviously a module over $R=  k[x_1, \cdots, x_s]/ (x_1^{d_1} + \cdots + x_s^{d_s})$. Moreover,
as $R$-modules, 
 \[ M/TM \cong  k[x_1, \cdots, x_s]/ (x_1^q, \dots, x_s^q, x_1^d + \cdots + x_s^{d_s}) 
     \cong  R/ \pp m^{ [q] } , \]
and  $\ell_R(R/ \pp m^{[q]} ) = \dim_{\kappa} (M/TM)$.
As noted earlier, $\alpha([M]) = \dim_{\kappa} (M/TM)$. 
Hence the Hilbert-Kunz function can be obtained via $\alpha$ as 
\begin{equation}\label{eq_alpha} \begin{array}{lcl}
\ell_R(R/ \pp m^{[q]} ) = \dim_{\kappa} (M/TM)  &  =  & \alpha( [M] )  
   = \alpha( [ M_1\otimes_{\kappa} \cdots \otimes_{\kappa} M_s] )   \\
 & = &  \alpha( [M_1]  \cdots  [M_s]) \\
 & = &  \alpha \left( {  \prod_{i=1}^d } (a_i \delta _{k_i+1} + (d_i-a_i) \delta _{k_i} ) \right) .
 \end{array}
\end{equation}
Han and Monsky developed an intricate combinatorial method of computing 
the indecomposable summands in the decomposition of tensor products of basis elements 
({i.e.,} products of $\delta_j$'s) in $\Gamma$ to calculate the value of $\alpha$ in the last equality in 
(\ref{eq_alpha}). This yields a method for computing Hilbert-Kunz functions.  
Using this, they also proved that the Hilbert-Kunz function of the rings in this family is eventually periodic and that the multiplicity is a rational number~\cite[Theorem~5.7]{HaM}.  

We note that the eventually periodic functions of the family in this subsection are not in the form of quasipolynomials as will be discussed in Section~\ref{affineEhrhart}. This can be seen from the following examples in \cite{HaM}.

\begin{example}\cite[p.~135]{HaM}\label{ex_HaM} \\
(1) $R = \kappa[x_1, x_2, x_3, x_4]/(\sum_i x_i^4)$ where $\kappa = \mathbb Z/5 \mathbb Z$, 
then $\hk R = \frac{168}{61} \cdot 5^{3e} - \frac{107}{61} \cdot 3 ^e$. \\
(2) $R = \kappa[x_1, \dots, x_5]/(\sum_i x_i^4)$ where $\kappa = \mathbb Z/3 \mathbb Z$,
then $\hk R = \frac{23}{19} \cdot 3^{4e} - \frac{4}{19} \cdot 5 ^e$.
\end{example}

Using the technique developed by Han and Monsky~\cite{HaM}, 
Chiang and Hung~\cite{ChH1998} extended the result from diagonal hypersurfaces to rings of the form 
$\kappa[x_1, \dots, x_s]/(g)$ where $g = \sum_{i=1}^m ({\bf X}_i)^{d_i}$ and ${\bf X}_i$ 
is a product of elements of a subset of $\{x_1, \dots, x_s\}$ such that at least one ${\bf X}_i$ is a single variable. 

{\bf \em $p$-Fractals.} 
Building on the theory of representation rings, Monsky and Teixeira~\cite{MoT2004, MoT2006, Tei2002} 
develop the theory of $p$-fractals which takes them beyond diagonal hypersurfaces.  
The theory of $p$-fractals provides a means to understand and allows the computations of Hilbert-Kunz series in may situations. 
Here we give a brief overview of the results, but do not delve into this method due to its technical complexity. 

First, recall that {\em Hilbert-Kunz series} is the generating function of Hilbert-Kunz function
\[ \hks R = \sum_e \hk{R}{e} \cdot t^e . \] 
Seibert~\cite{Sei1997} first proves that if $R$ is Cohen-Macaulay of finite Cohen-Macaulay type, and $R$ is 
$F$-finite (a finitely generated module via the Frobenius morphism), then $\hks R$ is rational, 
{i.e.}, a quotient of two polynomials, and that $e_{HK}(R)$ is a rational number. 
Hilbert-Kunz series has also been studied in \cite{Mon5, Tei2002, MoT2006}.

The combined techniques of representation rings and $p$-fractals enable Monsky and Teixeira to prove the rationality of 
Hilbert-Kunz series for a large family of rings of the form of a quotient of a power series ring by a principal ideal: $\kappa [\![ x_1, \dots, x_s ]\!]/(g)$ for a finite field $|\kappa| $ and certain power series $g$ \cite[Theorem~4.4]{MoT2006}. 
Using this, in some examples, the Hilbert-Kunz multiplicities can be calculated and appear to be rational numbers  \cite[Section~5]{MoT2006}. 
Furthermore,  the Hilbert-Kunz functions are eventually periodic for the case 
$s=3$, $g = x_3^D-h(x_1, x_2)$ and $h $ a nonzero element in the maximal ideal of $\kappa [\![ x_1, x_2 ]\!]$ \cite[Theorem~6.11]{Mon5}. (See also the ending comment in \cite[p.~255]{MoT2006} regarding the finiteness condition on the field $\kappa$ and the vanishing of  the $(p^e)$-term of the Hilbert-Kunz function. We will address these interesting points in Subsection~\ref{sub_sheaf} after the proof of Theorem~\ref{geoBrenner} and in Subsection~\ref{sub_Chern}, respectively.)
In \cite[Theorem, p.~351]{Mon5}, Monsky gives a concrete description of the Hilbert-Kunz function for the case 
$g = x_3^D - x_1x_2(x_1 + x_2)(x_1 + \lambda x_2)$, for $\lambda \neq 0$ in $\kappa$, and $p \equiv \pm 1 \pmod D$.

Next we present an attempt to define Hilbert-Kunz multiplicity in the characteristic zero case and 
report some concrete progress based on Han-Monsky's technique of representation rings. 

\begin{application}\label{limitHK} 
The notion of limit Hilbert-Kunz multiplicity arises by reducing a ring of characteristic zero modulo primes $p$. 
Precisely, if $R$ is a $\mathbb Z$-algebra essentially of finite type over $\mathbb Z$ and $I$ is an ideal, let $R_p$ be a reduction of $R$ mod $p$ and $I_p$ the extended ideal. If $\ell (R_p /I_p)$ is finite and nonzero for almost all $p$, then one can define 
\[ e_{HK}^{\infty} (I,R)  : = \displaystyle{ \lim_{p \rightarrow \infty} } e_{HK} (I_p, R_p) \] 
if this limit exists. The above is called the {\em limit Hilbert-Kunz multiplicity} of $I$. 
The existence of the limit is not known except in a few cases. 
These include when $e_{HK} (I_p, R_p)$ is constant 
for almost all $p$ as in \cite{Br1, BuC, Co1996, Et, Mon2, Par1994, Wat, WY2},
and when non-constant, for homogenous coordinate rings of smooth projective curves
by Trivedi~\cite{Tr2007} and for diagonal hypersurfaces by Gessel and Monsky in 
\cite{GeM2010}. The example below illustrates the last case. 

Based on Han and Monsky's scheme, Chang~\cite{ShTCh1993} computes precisely the Hilbert-Kunz function of 
the ring $R=\kappa[x_1, x_2, x_3, x_4]/(x_1^4 + x_2 ^ 4 + x_3^4 + x_4^4)$ with respect to the maximal ideal $(x_1, 
\dots, x_4)$ for $\kappa=\mathbb Z/p \mathbb Z$. The Hilbert-Kunz multiplicity in this case is also calculated in Gessel and Monsky~\cite{GeM2010} to be
\[ e_{HK} (R) = \frac{8}{3} \left( \frac{2p^2 + 2p +3}{2p^2 + 2p + 1} \right), p \equiv 1 \pmod 4; \text{ and } 
\frac{8}{3} \left( \frac{2p^2-  2p +3}{2p^2 - 2p + 1} \right),  p \equiv 3 \pmod 4 \] 
In this example, $e_{HK}(R)$ obviously depends on the characteristic $p$ and yet one has  
\[ e_{HK}^{\infty}(\pp m, R) =  
\displaystyle{ \lim_{p \rightarrow \infty} } \, \frac{8}{3} \left( \frac{2p^2 \pm  2p +3}{2p^2 \pm 2p + 1} \right) = \frac{8}{3} \] 
which is independent of $p$. 

Assuming that $e_{HK}^{\infty}(I,R) $ exists, one can also ask for any fixed $e \geq 1$ 
if the following equality holds, 
\begin{equation}\label{q_limitHK}
 e_{HK}^{\infty}(I, R)  = \displaystyle{ \lim_{p \rightarrow \infty}  \frac{\ell(R_p/ I_p^{ [p^e] } ) }{ (p^e)^d } } .
\end{equation}
Brenner, Li and Miller~\cite{BLM2012} provide an affirmative answer for the case of homogenous coordinate rings of smooth projective curves. Their proof is based on the formula of the Hilbert-Kunz multiplicity provided by the Harder-Narasimhan filtrations as described in Subsection~\ref{sub_sheaf} and the approach in \cite{Tr2007} 
which handles the case of larger values of $e$. 
Then building on the proof of \cite{GeM2010}, it is proved  in \cite{BLM2012} that  the above question has affirmative answer for coordinate rings of diagonal hypersurfaces. 
More recent developments can be found in \cite{Tr2017, MoT2019, Tr2019, Smi2020, TrW2021, PTYpre2021}. 
\end{application}

\subsection{Via The Divisor Class Group}\label{sub_classgp}
Assume that $(R, \pp m)$ is an excellent normal local domain of $\dim R=d$ with a perfect residue field.
Here we sketch the main ideas of Huneke, McDermott and Monsky~\cite{HMM} 
in proving that the Hilbert-Kunz function of $M$ with respect to an $\pp m$-primary ideal $I$  has  
the following form as in (\ref{fcn:HMM}) 

\begin{equation}\label{fcn:HMMq}
\hk{M,I}{e} = \alpha \, (p^e)^{d} + \beta \,  (p^e)^{d-1}+  \mathcal O(  (p^e)^{d-2}).
\end{equation}

We remark that in general the $d$ in (\ref{fcn:HMMq}) can be replaced by $\dim M$ as noted in the statement after Theorem~\ref{thmMonsky}. For the purpose of discussions in this subsection, we fix $d$ to be $\dim R$. 
Thus the $\alpha$ in (\ref{fcn:HMMq}) is zero for those $M$ with $\dim M <d$. And similarly, $\beta = 0$ if $\dim M < d-1$. 

Below we state two key lemmas. Lemma~\ref{HMMkey} is from \cite{HMM} and is a consequence of (\ref{boundlength})  that estimates the length of higher homology modules mentioned after Corollary~\ref{regular}. Lemma~\ref{estimate} is implicit in \cite{HMM} and is a crucial fact regarding the convergence rate of a sequence. 

\begin{lemma}\cite[Lemma~1.1]{HMM}\label{HMMkey} 
Let $R$ be a local ring of positive characteristic $p$ and $I$ and $\pp m$-primary ideal of $R$. 
If $T$ is a finitely generated torsion $R$-module with $\dim T=u$, then $\ell( \tor_1^R(R/I^{[p^e]}, T) ) = O( (p^e)^{u} ) $.
\end{lemma}

\begin{lemma}\label{estimate}
Let $s > t$ be fixed positive integers. 
Let $\{ \eta_e \}_{e=1}^{\infty}$ be a sequence of real numbers that satisfies a recurrent relation  
$\eta_{e+1} = p^s \eta_e + O((p^e)^{t})$ for $e \gg 1$.
Then there exists a real number $ a $ such that 
$\eta_e = a \cdot (p^e)^s + O( (p^e)^{t})$ for $e \gg 1$.
\end{lemma}

\begin{proof} 
Set $\rho_e =\frac{ \eta_e}{(p^e)^s}$. Then  by straightforward 
computation, one sees $ \rho_{e+1} = \rho_{e} + O( ( \frac{1}{ p^e })^{s-t} )$ which means 
 $| \rho_{e+1} -  \rho_{e} | \leq C \cdot  (\frac{1}{p^e})^{s-t} $ for some constant $C$ and $e \gg 0$. 
Let $m  > n$ be integers. Then 
\[\begin{array}{lcl}
 |\rho_{m} - \rho_{n}|  & \leq &  |\rho_{m} - \rho_{m -1}| + \cdots + |\rho_{n+2} - \rho_{n+1}| 
  +  |\rho_{n+1} - \rho_n |. \\
  & \leq & C [ (\frac{1}{p^{m-1} })^{s-t} + \cdots + (\frac{1}{p^{n+1}})^{s-t} + (\frac{1}{p^n})^{s-t} ] \\
  & = & C \cdot (\frac{1}{p^{s-t}})^n [ 1+ (\frac{1}{p^{s-t}}) + \cdots + (\frac{1}{p^{s-t}})^{m-n-1} ] \\
  & = & C \frac{p^{s-t}}{p^{s-t} -1} [ (\frac{1}{p^{s-t}})^n - (\frac{1}{p^{s-t}})^m ]
\end{array}
\]
which can be made as small as possible for large enough $m$, $n$. 
Hence $\{\rho_e\}$ is a Cauchy sequence and thus it converges to some real number $a$. Furthermore, if we 
fix $n=e \gg 1$, then 
\[
 \lim_{m\rightarrow \infty }|\rho_{m} - \rho_{e}| \leq  \lim_{m\rightarrow \infty }
  C \frac{p^{s-t}}{p^{s-t} -1} [ (\frac{1}{p^{s-t}})^e - (\frac{1}{p^{s-t}})^m ] 
  = C \frac{p^{s-t}}{p^{s-t} -1}  (\frac{1}{p^{s-t}})^e .
\]
Note that 
 $ \displaystyle{\lim_{m\rightarrow \infty } }|\rho_{m} - \rho_{e}| = 
   |a - \rho_e|$. Hence $|a - \rho_e| \leq  O(  (\frac{1}{p^{s-t}})^e ) $ which implies 
  $\eta_e = a (p^e)^s + O( (p^e)^t ) $ as desired. 
\end{proof} 

The work in \cite{HMM} shows that the value of the second coefficient $\beta$ is intimately related to the divisor class of the module, denoted by $c(\cdot)$.  
The proof of (\ref{fcn:HMMq}) goes through the following three main steps.
{\em Step}~1 deals with torsion free modules and compares them with free modules of the same rank. 
{\em Step}~2 applies the outcome of {\em Step}~1 to $^1\!R$ to obtain $\hk{R,I}{e}$. 
{\em Step}~3 reduces arbitrary modules to torsion free ones and then applies {\em Steps} 1, 2 and Lemma~\ref{HMMkey} to obtain $\hk{M,I}{e}$ for arbitrary $M$. 

{\bf \em Step 1.} We focus on torsion free modules and prove that
if $M$ has rank $r$, then there exists a real constant $\tau (M)$ such that 
$\hk{M,I}{e} = r \, \hk{R,I}{e} + \tau(M)  (p^e)^{d-1} + \mathcal O(  (p^e)^{d-2})$. 
That is, up to $\mathcal O(  (p^e)^{d-2})$, the Hilbert-Kunz function of $M$ differs from that of a free module of the same rank by a constant multiple of $ (p^e)^{d-1}$ for $e \gg 1$. 
We give a brief outline of the proof for this statement.
For a nonzero ideal $J$ (equivalently torsion free module of rank one), if its divisor class $c(J)=0$, then $R/J$ is torsion and $\dim (R/J) \leq d-2$. 
The short exact sequence $0 \rightarrow J \rightarrow R \rightarrow R/J \rightarrow 0$ implies that 
$\hk{J,I}{e}= \hk{R,I}{e} - \hk{R/J, I}{e} + \ell( \tor_1^R(R/I^{[p^e]}, R/J) )$. 
So using \cite[Lemma~1.2]{Mon1} and Lemma~\ref{HMMkey} above, one can prove that $\hk{J,I}{e} = \hk{R,I}{e} + O( (p^e)^{d-2})$ (\cite[Lemma~1.2]{HMM}). 
A similar result for a torsion free module $M$ of rank $r$ with $c(M)=0$ can be achieved, namely
$\hk{M,I}{e} = r \hk{M,I}{e} + O( (p^e)^{d-2})$ (\cite[Theorem~1.4]{HMM}). 
Next, if $M$ and $N$ are torsion free and $c(M)=c(N)$, then their Hilbert-Kunz functions are equal up to $O( (p^e)^{d-2})$ and also $\ell( \tor_1(R/I^{[p^e]}, M) = O( (p^e)^{d-2})$ 
\cite[Lemma~1.5]{HMM}. (This latter statement will be needed in {\em Step~3}.) 
We note that since two modules are often fit into the same short exact sequence in order to compare their Hilbert-Kunz functions (as just done for $J$ and $R$ in the above), bounding the length of the $\tor$-functor within a desired range becomes crucial for the success of the argument. Finally for an arbitrary torsion free module of rank $r$, one considers 
$\delta_e= \hk{M,I}{e} - r \hk{R,I}{e}$ and proves that this difference satisfies a recurrence relation: 
$\delta_{e+1} =  (p^e)^{d-1} \delta_e + O( (p^e)^{d-2} )$ \cite[Theorem~1.8]{HMM}. Thus
by Lemma~\ref{estimate}, there exists a real number, denoted by $\tau(M)$, such that 
\[ \delta_e = \tau(M) (p^e)^{d-1} + O( (p^e)^{d-2}) . \] 
By the definition of $\delta_e$, we have 
\[ \hk{M,I}{e} = r \,\hk{R,I}{e} + \tau(M)  (p^e)^{ d-1} +  O( (p^e)^{d-2}) \]
where the number $\tau(M)$ depends only on the class of $M$ and is additive. In fact $ M \rightarrow \tau(M)$ gives a well-defined map on the divisor class group of $R$ to $\mathbb R$. In particular, $\tau(M) =0$ if $c(M)=0$ ({\em c.f.} \cite[Theorem~1.9]{HMM}\,\cite[Theorem~4.1]{ChK1}).

{\bf \em Step 2.} In this step, we prove $\hk{R,I}{e}$ has the desired form (\ref{fcn:HMMq}) by taking $M= {^1\!R}$ 
and repeating a similar approximation as in {\em Step}~1.  
In fact, $^1\!R$ is a finitely generated $R$-module by the hypothesis of $R$. 
Notice that 
\[ 
 \hk{^1\!R,I}{e} = \ell_R({^1\!R}/I^{[p^e]} {^1\!R} ) = \ell_R( R/I^{[p^ {e+1} ]} ) = \hk{R,I}{e+1}  .
\]
On the other hand, since $R$ is a domain and $R$ is $F$-finite, ${^1\!R}$ is a torsion free $R$-module of rank $p^d$. 
Thus by {\em Step}~1, we have 
\begin{equation}\label{eq_1}
 \hk{R,I}{e+1} = \hk{^1\!R,I}{e} = p^d \hk{R,I}{e} + \tau({^1R})  (p^e)^{d-1} + O( (p^e)^{d-2} ) .
\end{equation}
Let $\tau = \tau(^1\!R)$. We set
\[ u_e = \hk{R,I}{e} - \beta (p^e)^{d-1} \] 
for some $\beta$ whose value will be clear in the following. 
Using (\ref{eq_1}), we calculate
\[ \begin{array}{lcl}
u_{e+1} - p^d u_e & =  &
   \left [  \hk{R,I}{e+1} - \beta  \, p^{d-1} (p^{e})^{d-1} \right ]  -  p^d \left [ \hk{R,I}{e} - \beta\, (p^e)^{d-1} \right ]  \\ 
& = & \left [ p^d \hk{R,I}{e} + \tau\, (p^e)^{d-1} + O( (p^e)^{d-2} )\right ]  \\
   &  &  -  \left [ p\, \hk{R,I}{e} + (p^d - p^{d-1}) \beta \,(p^e)^{d-1} + O( (p^e)^{d-2} \right ] \\
& = & [ \tau + (p^d - p^{d-1}) \beta] \, (p^e)^{d-1} + O( (p^e)^{d-2} ) .
\end{array} \] 
Now by setting $\beta = \displaystyle{ \frac{{\tau}}{p^{d-1} - p^d } }$, 
the first term following the last equality in the display above vanishes and 
hence $u_{e+1} - p^d u_e = O( (p^e)^{d-2} )$. 
Thus applying Lemma~\ref{estimate}, there exists a real number $\alpha$  such that 
\[ u_e = \alpha\, (p^e)^{d} + O( (p^e)^{d-2} ) . \] 
Recovering $\hk{R,I}{e}$ from $u_e$, we then have the desired form
\[ \hk {R,I}{e} = \alpha \,(p^e)^{d} + \beta\, (p^e)^{d-1} + O( (p^e)^{d-2} . \]
Comparing to Monsky's original result for $\hk{R,I}{e}$, the leading coefficient $\alpha$ must be the Hilbert-Kunz multiplicity $e_{HK}(R,I)$ as expected.

{\bf \em Step 3.} For an arbitrary module $M$, let $T$ be the submodule of torsion elements in $M$. We have a short exact sequence 
$0 \longrightarrow T \longrightarrow M \longrightarrow M/T \longrightarrow 0 $
where $M/T$ is a torsion free module, denoted $M'$. Tensoring the sequence by $R/\pp m^{[p^e]}$, 
we obtain a long exact sequence
\[ \cdots \longrightarrow \tor_1(M', R/\pp m^{[p^e]}) \longrightarrow T/ \pp m^{[ p^e]} T \longrightarrow M/ \pp m^{[p^e]} M 
\longrightarrow M'/\pp m^{[p^e]}M' \longrightarrow 0  . \]
By \cite[Lemma~1.5]{HMM}, we know  $\ell( \tor_1( M', R/\pp m^{[p^e]})  ) = \mathcal O( (p^e)^{d-2}) $. Thus
\[ \begin{array}{lcl} 
\hk{M,I}{e}  & = & \hk{M',I}{e} + \hk{T,I}{e} + \mathcal O(  (p^e)^{d-2})  \\ 
           & = & r\,  \hk{R,I}{e} + \tau(M') \,  (p^e)^{d-1} + \hk{T,I}{e} + \mathcal O( (p^e)^{d-2})  \end{array} \] 
           where $r = \rank M'$. 
The proof is completed by noticing that $\hk{R,I}{e}$ has the desired form from {\em Step}~2 and since $\dim T \leq d-1$, 
we have $\hk{T,I}{e} = \beta(T) \,  (p^e)^{d-1} + \mathcal O( (p^e)^{d-2})$ for some $\beta (T) \in \mathbb R$ 
by \cite[Lemma~1.2 and Theorem 1.8]{Mon1}.  \hfill $\Box$

Instead of the normal (R1) + (S2) condition, Kurano and the author consider a weaker condition  in \cite{ChK1}. 
The ring $R$ satisfies (R1$'$) if the localization of $R$ is a field at any prime ideal of Krull dimension $d$, 
and is a DVR at any prime ideal of Krull dimension $d-1$. 
The (R1$'$)  condition is similar to, but not the same as the usual (R1). 
It can be shown that (\ref{fcn:HMM}) holds for excellent local rings that satisfy (R1$'$) but are not necessarily integral domains. 
The proof is done by reducing to the assumption that $R$ is a normal domain (\cite[Theorem~3.2]{ChK1}). (See also \cite{HoY} for a different approach.) 

On the other hand,  
we observe that each step of the proof in \cite{HMM} just outlined is interesting in its own right.  
These steps work more generally than just in a normal setting. 
If $R$ is not normal, the divisor class group is no longer well-defined. 
The immediate challenge is the description of $\tau$ that leads to the second coefficient $\beta$. 
The Chow group is a natural substitute for the divisor class group in the non-normal case.
We now describe how to replace divisor classes of modules by cycle classes.
For a finitely generated module $M$, there always exists a finite filtration of submodules, called a {\em prime filtration},
such that the quotient of two consecutive submodules is isomorphic to a quotient of $R$ by a prime ideal. 
Let $p_1, \cdots, p_s$ be the prime ideals of codimension 0 or 1 that occur in such a prime filtration.
Then these prime ideals define a cycle class  $[M] = [A/p_1] + \cdots + [A/p_s]$ in the {\em Chow group} $A_*(R)$ (see Subsection~\ref{sub_Chern} for the definition).
Since   $A_*(R)= \oplus_{i=0}^{d} A_i(R)$ according to its prime ideal generators, 
we have  $[M] \in A_d(R) \oplus A_{d-1}(R)$.
Even though a collection of such prime ideals $\pp p_i$'s may not be unique for prime filtrations are not unique, 
the class $[M]$ in the Chow group is independent of the choice of filtration and is additive as proved in \cite[Theorem~1 and Corollary~1]{C1}. This implies that each finitely generated module $M$ has a unique cycle class defined in the Chow group obtained by a filtration.  Hence the definition of the map $\tau$ can be extended to a homomorphism from 
$ A_d(R) \oplus A_{d-1}(R)$ to $\mathbb R$.
 The proof in \cite{HMM} is extended step by step to the case where $R$ is an integral domain satisfying (R1$'$) in \cite[Section~5]{ChK1}. 
The cycle classes of the modules affect $\tau$ and  $\beta$ in the same way as in the normal situation. 
This extension also inspires the consideration of an additive error of the Hilbert-Kunz function which is not additive on short exact sequences. 
With the $\tau$ map mentioned above, one sees that the additive error always arises from torsion submodules 
and is determined by their classes in the Chow group \cite[Section~4]{ChK1}. 

The vanishing of the second coefficient $\beta$ has been of interest since its discovery. 
We will return to this topic in Subsection~\ref{sub_Chern}.  See also Theorem~\ref{planecurve} for an example where
$\beta$ does not vanish. 

The divisor class group will appear again in Subsection~\ref{sub_BG} when we review the BG decomposition of affine semigroups~(\cite{Bru1}). The modules $M$ and $M/T$ define the same divisor class (or cycle class) in the divisor class group ({\em resp.} Chow group). From the above sketch, we notice that the leading coefficient of $\hk{M,I}{e}$ is determined by the rank of $M$ but the second coefficient $\beta$ (or the additive error) depends on the class of $M$ ({\em resp.} classes of torsion submodules). So in order to understand the second coefficient, or the remaining terms of the Hilbert-Kunz function, the divisor class group (or Chow group) cannot be overlooked (see Remark~\ref{rmk:divisor}).

\subsection{Via Sheaf Theory}\label{sub_sheaf}

In this subsection, $R$ is a standard graded $\kappa$-algebra. 
Sheaf theoretic approaches were first considered independently by 
Fakhruddin and Trivedi~\cite{FaT, Tr2005}, and by Brenner~\cite{Br1, Br2}. 
The general idea is that $\hk{R, I}{e}$ is identified as the alternating sum of the lengths of sheaf cohomology modules. 
Then they carefully study the sheaves occurring in the sequences arising from the resolution of $R/I$ to 
describe $\hk{R,I}{e}$.

In this subsection, we present one of these approaches, following the argument in \cite[Section~6]{Br2} 
(equivalent to that by Trivedi in \cite{Tr2005}). 
In that work, Brenner applies the theory built for locally free sheaves on smooth projective curves to obtain 
the Hilbert-Kunz functions of two-dimensional normal domains that are standard graded $\kappa$-algebras 
with an algebraically closed field $\kappa$ of positive characteristic $p$. 
In this case, the Hilbert-Kunz function has the following form
\begin{equation}\label{fcn:Brenner} 
 \hk{R}{e}=\hkq{R}{q} = \alpha \, q^2 + \gamma(q) 
\end{equation}
where $\alpha$ is a rational number(\cite{Br1, Tr2005}), and $\gamma(q)$ is a bounded function taking on 
rational values and is eventually periodic when $\kappa$ is the algebraic closure of a finite field~\cite{Br2}. 
For detailed background in this subsection, one can also consult Brenner's lecture \cite{Br3lecture}.

Let $R_+$ denote the graded maximal ideal of $R$ and $I$ be a graded ideal primary to $R_+$ generated by 
$f_1, \dots, f_s$ of degree $d_1, \dots, d_s$ respectively. We consider $Y=\proj R$ which is assumed to be a smooth
projective curve. Equivalently, $R$ is normal and $\dim R=2$. With these, one obtains the following short exact sequence of coherent sheaves on $Y$
\begin{equation}\label{seq_syzygy}
 0 \longrightarrow Syz(f_1, \dots, f_n)(m) \longrightarrow \oplus_{i=1}^n \mathcal O_Y(m-d_i) 
 \stackrel{f_1, \dots, f_n}{\longrightarrow} \mathcal O_Y(m) \longrightarrow 0 
\end{equation}
where $m \in \mathbb Z$ indicates the twist of the structure sheaf $\mathcal O_Y$ and $Syz(f_1, \dots, f_n)$ is known as the {\em syzygy sheaf} (or {\em syzygy bundle} if locally free). 

We explain below why  $Syz(f_1, \dots, f_n)$ is a locally free sheaf. 
Let $K$ be an $R$-module such that the sequence 
$0 \rightarrow K \rightarrow R^n \stackrel{f_1, \dots, f_n}{\longrightarrow} R \rightarrow R/I \rightarrow 0 $
 is exact, or equivalently, $R/I$ is the zero-th cohomology of the following complex
\begin{equation}\label{seq_complex}
0 \longrightarrow K \longrightarrow \oplus_{i=1}^n R(-d_i) \stackrel{f_1, \dots, f_n}{\longrightarrow} R \longrightarrow  0 .
\end{equation}
Since $I$ is $R_+$-primary, $R/I$ is only supported at $R_+$. Therefore for any prime ideals $\pp p$ not equal to 
$R_+$, the localization $(R/I)_{\pp p} =0$. This shows that the above complex (\ref{seq_complex}) is exact locally at every prime ideal 
$\pp p \in \proj R$. 
That is, the following sequence is exact and it obviously splits since $R_{\pp p}$ is free
\begin{equation}\label{seq_complexlocal}
0 \longrightarrow K_{\pp p} \longrightarrow \oplus_{i=1}^n R_{\pp p}(-d_i) \longrightarrow R_{\pp p} \longrightarrow  0 .
\end{equation}
Hence, as a a direct summand of a finitely generated free module over a local ring, $K_{\pp p}$ is a free module. 
Taking the sheafification of (\ref{seq_complex}), we have the following exact sequence of locally free sheaves on $\proj R$
\begin{equation}\label{seq_sheaf}
 0 \longrightarrow \widetilde K \longrightarrow  \oplus_{i=1}^n \mathcal O_Y(-d_i) 
  \longrightarrow \mathcal O_Y \longrightarrow  0 
 \end{equation}
and $\widetilde K$ is precisely the syzygy bundle $Syz(f_1, \dots, f_n) $ under consideration.
The complex (\ref{seq_syzygy}) is obtained by twisting (\ref{seq_sheaf}) by an integer $m$; equivalently by tensoring with $\mathcal O(m)$.  Since (\ref{seq_complexlocal}) is split exact on stalks, so is (\ref{seq_syzygy}).
Thus the complex remains exact when applying any additive functor, including the tensor product.

Next we consider the {\em absolute} Frobenius morphism $\Ff: Y \rightarrow Y$ which is the identity on the points of $Y$ and 
furthermore, on every local ring of sections, it is the Frobenius homomorphism $f$. The Frobenius pull-back of a sheaf of 
modules on $Y$ is obtained by base change along the Frobenius homomorphism, {i.e.}, 
\[ \Ff^*(\widetilde M) = \widetilde{ {^1\!R} \otimes_R M} = \widetilde{ \Ff(M)} . \]
As (\ref{seq_syzygy}) is split exact on stalks, the Frobenius pull-back of this sequence remains exact. 
The resulting locally free sheaves are of the same rank twisted by the appropriate degree, 
but the multiplication maps are raised to the $p$-th power. 
One may iterate this $e$ times and the sequence remains exact:  
\begin{equation}\label{sheaf} 
0 \longrightarrow ( Syz(f_1^{p^e}, \dots, f_n^{p^e}) )(m) \longrightarrow \oplus_{i=1}^n \mathcal O_Y(m- p^e d_i) 
 \stackrel{f_1^{p^e}, \dots, f_n^{p^e}}{\longrightarrow} \mathcal O_Y(m) \longrightarrow 0 . 
\end{equation}
Taking global sections is a left exact functor.  Since $R$ is normal, the global sections of twists of the structure sheaf can be
realized as the graded pieces of degree $m$:
\[ 
0 \longrightarrow \Gamma \left (Y, (Syz(f_1^{p^e}, \dots, f_n^{e})) (m) \right) \longrightarrow  \oplus_{i=1}^n R_{m-p^ed_i} 
 \stackrel{f_1^{p^e}, \dots, f_n^{p^e}}{\longrightarrow} R_m \longrightarrow \cdots  .
\]
The cokernel at $R_m$ is 
the graded piece of degree m in $ R/ I^{[p^e]}$.  By definition, the Hilbert-Kunz function is 
\begin{equation}\label{HKsheaf}
\hk{R,I}{e} = \ell \left (  R/ I^{[p^e]} \right ) = \sum_{m \geq 0}  \ell \left(  ( R/ I^{[p^e]} )_m \right) 
=  \sum_{m \geq 0}  \dim_{\kappa} ( \Gamma(Y, \mathcal O_Y(m)) / ( f_1^{p^e}, \dots, f_n^{p^e} ) ) .
\end{equation}   
Thus for each $m \geq 1$, we have 
\begin{equation}\label{HKhomology}
\ell \left(  ( R/ I^{[p^e]} )_m \right) = 
h^0( \mathcal O_Y(m) ) - \sum_{i=1}^n h^0( \mathcal O_Y(m-p^e d_i) ) + h^0 ( Syz(f_1^{p^e}, \dots, f_n^{p^e})(m) ), 
\end{equation} 
where $h^0(\cdot)$ denotes the $\kappa$-dimension of the $0$-th cohomology module which is
the module of global sections of the  sheaf in the argument.

We remark that (\ref{HKsheaf}) is a finite sum since the sum is only defined for $m \geq 0$.
In addition, the alternating sum in (\ref{HKhomology}) is zero for $m \gg 0$ 
due to Serre's vanishing theorem. 
In fact, by \cite[Lemma~9.4]{Br3lecture}, it suffices to  consider $m$ within the certain range as described in the summation presented 
in (\ref{global}) in order to compute the right hand side of (\ref{HKsheaf}) .

To analyze the terms in (\ref{HKhomology}), 
one reduces to the situation of considering semistable locally free sheaves. 
This is done by carefully applying the strong Harder-Narashimhan filtration on locally free sheaves.  
To further explain the concepts, we first recall some definitions and their properties from 
\cite{Br1, Br2} (see also \cite[Chapters~5 and 9]{Br3lecture}). 
In the remainder of this subsection, $\mathcal S$ denotes a locally free sheaf on $Y$ of rank $r$. 
The {\em degree} and {\em slope} of $\mathcal S$ are defined by 
$\deg \mathcal S: = \deg \bigwedge^r (\mathcal S)$ and $\mu(\mathcal S) := \deg{\mathcal S}/r$. 
The slope has the property that $\mu(\mathcal S_1 \otimes \mathcal S_2) = \mu(\mathcal S_1) + \mu(\mathcal S_2)$. 
A locally free sheaf $\mathcal S$ is {\em semistable} if $\mu(\mathcal T) \leq \mu(\mathcal S)$ for every locally free subsheaf $\mathcal T \subseteq \mathcal S$. 
For every locally free sheaf on $Y$, there exists a unique {\em Harder-Narasimhan} filtration. This is a finite filtration of locally free subsheaves
$\mathcal S_1 \subset \mathcal S_2 \cdots \subset \mathcal S_t = \mathcal S$ such that the quotients 
$\mathcal S_k/\mathcal S_{k-1}$ are semistable and of decreasing slope: 
$\mu( \mathcal S_k/\mathcal S_{k-1} ) > \mu( \mathcal S_{k+1}/\mathcal S_{k} )$. 
Naturally the largest and smallest numbers in the sequence are called  the {\em maximal} and {\em minimal slopes} of $\mathcal S$. 
They are denoted by $\mu_{max}(\mathcal S)$ and $\mu_{min}( \mathcal S )$ respectively.
The sheaf $\mathcal S$ is semistable if and only if $\mu(\mathcal S) = \mu_{min} (\mathcal S)= \mu_{max}(\mathcal S)$. 

The Frobenius pull-back of a semistable locally free sheaf $\mathcal S$ is not necessarily semistable. 
Let $\mathcal S^q$ denote  the $q$-th Frobenius pull-back of $\mathcal S$. 
A locally free sheaf $\mathcal S$ is said to be {\em strongly semistable} if the pull-back of $\mathcal S^q$ is semistable for any $e\geq 1$.

The existence of a filtration with such nice factors is due to a theorem of Langer~\cite{La2004} (see also \cite{Br2}): there exists a Frobenius power $q$ such that 
the quotients in the Harder-Narasimhan filtration of the pull back $\mathcal S^q$ are strongly semistable. This is called the {\em strong Harder-Narasimhan filtration}  (of $\mathcal S^q$), denoted by
\begin{equation}\label{strong HN} 
0 \subset (\mathcal S^q)_1 \subset \cdots \subset (\mathcal S^q)_t = \mathcal S^q .
\end{equation}

We use $(\mathcal S^q)_k$ to indicate the members in the filtration of $\mathcal S^q$ 
and to distinguish it from the pull-back $(\mathcal S_k)^q$ of $\mathcal S_k$
whose quotient may not be semistable.
For $q' \geq q \gg 1$, we have the Harder-Narasimhan filtration of $\mathcal S^{q'}$
\begin{equation}\label{FfHN}
0 \subset (\mathcal S^q)_1^{q'/q} \subset \cdots \subset (\mathcal S^q)_t^{q'/q} = (\mathcal S^q)^{q'/q} = \mathcal S^{q'} .
\end{equation}
We observe that the quotient in (\ref{FfHN})
$(\mathcal S^q)_k^{q'/q} / (\mathcal S^q)_{k-1}^{q'/q} = \left ( ( \mathcal S^q)_{k} / (\mathcal S^q)_{k-1} \right )^{q'/q} $ 
is the pull-back of $ ( \mathcal S^q)_k / (\mathcal S^q)_{k-1} $. Thus these quotients are semistable since 
$ ( \mathcal S^q)_k / (\mathcal S^q)_{k-1} $ is strongly semistable which follows from Langer's Theorem. The slopes of these quotients are decreasing  since 
$\mu( ( \mathcal S^q)_k / (\mathcal S^q)_{k-1} ) > \mu ( ( \mathcal S^q)_{k+1} / (\mathcal S^q)_{k} ) $. 
Therefore (\ref{FfHN}) is a Harder-Narasimhan filtration of $\mathcal S^{q'}$, as a $(q'/q)$-th Frobenius pull-back of $\mathcal S^q$. 
Using (\ref{strong HN}), for $q \gg 1$, we 
consider the normalized slope of the quotients in the Harder-Narasimhan filtration of $\mathcal S^q$ and define
\[ {\bar \mu_k} = {\bar \mu_k}(\mathcal S) = \frac{ \mu( ( \mathcal S^q)_k  / (\mathcal S^q)_{k-1}  ) }{q} . \] 
Similarly to the above argument, for any $q' \geq q \gg 0$, if we take the Harder-Narasimhan filtration of
$\mathcal S^{q'}$ from (\ref{FfHN}), then \\
\[ 
\bar \mu_k(\mathcal S^{q'})  = \frac{ \mu( (\mathcal S^q)_k^{q'/q} / (\mathcal S^q)_{k-1}^{q'/q} )}{ q'} = 
  \frac{ (q'/q) \mu (( \mathcal S^q)_k / (\mathcal S^q)_{k-1}) }{q' }= 
 \frac{ \mu ( ( \mathcal S^q)_k / (\mathcal S^q)_{k-1}) }{ q} .
 \]
This shows that ${\bar \mu}_k( S^{q'})$ is independent of $q' \gg 1$. 

With the above, we define the {\em Hilbert-Kunz slope}: 
\[ \mu_{HK}(\mathcal S) = \sum_{k=1}^{t} r_k {\bar \mu}_k^2 \, , \]
where $r_k = \rk( (\mathcal S^q)_k ) / (\mathcal S^q)_{k-1} )$.
Note that Hilbert-Kunz slope is a positive rational number.
 In \cite{Br1}, the Hilbert-Kunz multiplicity of $R$ with respect to a homogeneous $R_+$-primary ideal 
 $I=(f_1, \dots, f_n)$ with $\deg f_i=d_i$ is expressed in terms of $\mu_{HK} $ as 
\[ e_{HK}(R,I) = \frac{1}{2 \deg Y} \left ( \mu_{HK} ( Syz(f_1, \cdots, f_n) ) - ( \deg Y )^2 \sum_{i=1}^n d_i^2 \right ) , \] 
where $\deg Y = \deg ( \mathcal O_Y(1) )$ is the degree of the curve. 

We fix some notation before stating the main results regarding the global sections of locally free sheaves in \cite{Br2}. 
Similarly to $h^0(\cdot)$, $h^1( \cdot)$ denotes the $\kappa$-dimension of the first homology of a sheaf.
We define $v_k = - {\bar \mu}_k/\deg Y$ and write 
$\lceil q v_k \rceil = qv_k + \pi_k$ with the eventually periodic function $\pi_k = \pi_k(q)$.
For simplicity, we sometimes drop the argument $q$.
Let $\sigma \leq v_1$ and $\rho \gg v_t$ be rational numbers. We set $\lceil q \rho \rceil = q\rho + \pi$ 

For $q=p^e \gg 1$, using Serre duality, the $\kappa$-dimension of the global sections of the twisted sheaf $\mathcal S^q(m)$ 
can be expressed in the following form 
\begin{equation}\label{global}
 \begin{array}{lcl}
  {\displaystyle \sum _{ m= \lceil q \sigma  \rceil } ^ { \lceil q \rho \rceil -1} h^0( \mathcal S^q (m) ) } & = & 
 \frac{q^2}{2 \deg Y} ( \mu_{HK}(\mathcal S) + 2 \rho \deg \mathcal (S) \deg Y + \rho^2 \rk (\mathcal S)  \deg(Y)^2 ) \\
      &&   + q \left ( \rho \rk(\mathcal S) + \frac{\deg \mathcal S}{\deg Y} \right )\left ( 1-g - \frac{\deg Y }{2} \right ) 
        + q \pi ( \deg \mathcal S + \rho \rk(\mathcal S) \deg Y) \\
      && + \rk( \mathcal S) \pi \left( ( \pi -1) \frac{\deg Y}{2} + 1 -g \right ) 
           - {\displaystyle \sum_{k=1}^{t} r_k \pi_k \left ( \pi_k - 1) \frac{\deg Y}{2} + 1 -g \right ) } \\
      && + {\displaystyle  \sum_{k=1}^{t} } \left( 
       {\displaystyle \sum _{ m= \lceil q v_k \rceil } ^ { \lceil q v_k \rceil + \lceil \frac{\deg \omega}{\deg Y} \rceil}  }
             h^1( ( (\mathcal S^q)_k / (\mathcal S^q)_{k-1} ) (m) ) \right ), 
 \end{array} 
\end{equation} 
where $g$ is the genus  and $\omega$ is the canonical sheaf of the curve $Y$ respectively. 

The proof of the above expression in \cite[Theorem~3.2]{Br2} utilizes the fact that the rank and degree are additive on short exact sequences and that the Hilbert-Kunz slope is additive on the quotients in the strong Harder-Harasimhan filtration.

A main theorem in \cite{Br2} is stated as follows.

\begin{theorem}[Brenner~\cite{Br2}, Theorem~4.2]\label{geoBrenner} 
Assume that $\kappa$ is an algebraically closed field of positive characteristic p. 
Let $Y$ be a smooth projective curve over $\kappa$. 
Let $ \mathcal S$ denote a locally free sheaf on $Y$ and $\mathcal S^q$ is the $q$-th Frobenius pull-back of $\mathcal S$. 
Let $\sigma \leq v_1$ and $\rho \gg v_t$ denote rational numbers. Then we have 
\[ \sum _{ m= \lceil q \sigma  \rceil } ^ { \lceil q \rho \rceil -1} h^0( \mathcal S^q (m) ) = 
\alpha q^2 + \beta(q) q+ \gamma(q) ,     \]
where $\alpha$ is a rational number and $\beta(q)$ is an eventually periodic function and $\gamma(q)$ is a bounded function (both with rational values). Moreover, if $\kappa$ is the algebraic closure of a finite field, then $\gamma(q)$ is also an eventually periodic function.  \end{theorem} 

We outline the proof of Theorem~\ref{geoBrenner} paying special attention to why 
$\gamma(q)$ is eventually periodic when the ground field $\kappa$ is the algebraic closure of a finite field. 

{\em Beginning of the sketch.}
The first statement on the summation of the global sections is a consequence of (\ref{global}) once the last summand is understood. One observes that all the values in (\ref{global}) including the Hilbert-Kunz slope $\mu_{HK}(\mathcal S)$ are rational numbers. Thus the leading coefficient $\alpha$ of $q^2$ is rational and $\beta(q)$ and $\gamma(q)$ are  both rational valued. The function $\beta(q)$ depends on $\pi$ which is a periodic function of $q \gg 1$. Hence it is eventually periodic.
 The boundedness of $\gamma(q)$ is a result of \cite[Lemma~4.1]{Br2} which proves the existence of an upper bound of the sum of 
 $h^1$ terms (\ref{global}).

Below, we point out two facts in the proof of \cite[Theorem~4.2]{Br2} that lead to 
the periodicity of $\gamma(q)$ under the assumption that $\kappa$ is an algebraic closure of a finite field.
For a given rational number $v$, write $m(q) = \lceil qv \rceil = qv + \pi(q)$ with $0 \leq \pi(q) <1$. 
Notice that $\pi(q)$ is an eventually periodic function. Let $\tilde q$ denote its period. 

{\bf Fact~1.} $\deg( \mathcal S^q ( m(q) )$ is eventually periodic with period $\tilde{q}$. 

We consider a subset $M$ of $\mathbb N$ of the type $M = \{ q_0 \tilde q ^{\ell} | \ell \in \mathbb N \}$ where 
$q_0=p^{e_0}$  for some $e_0$ and satisfies $1 \leq q_0  < \tilde q$. 

{\bf Fact~2.} $\mathcal S^{q \tilde q} ( m(q) ) = \mathcal S^q (m(q) ) ^{\tilde q} \otimes \mathcal O( ( - \tilde q + 1) \pi (q) ) $ for all $q \in M$. 

From {\bf Fact~1}, all $\mathcal S^q( m(q) )$ with $q \in M$ have the same degree. 
Furthermore, due to the fact that there are only finitely many semistable locally free sheaves with the same fixed degree on $Y$
defined over the finite field of which $\kappa$ is an algebraic closure,  
the following set 
\[ \mathscr S := \{  \mathcal S^{q } ( m(q) ) | q \in M \} \]
is a finite set. 
Then, {\bf Fact~2} gives a recursive relation as $q$ varies in $M$. 
Since $\mathscr S$ is a finite set, this relation leads to an eventually periodic pattern. 
Other sheaves, namely $\mathcal S^q(m(q))$ where $q \notin M$, can be identified as $\mathcal S^q( m(q) + s) = \mathcal S^q( m(q) )\otimes \mathcal O( s) $ for some $q \in M$ and $s$ taking values in a fixed bounded interval. 
Thus they will occur eventually periodically as some element from 
$\mathscr S$ tensored by $\mathcal O_Y(s)$ with $s$ in the same fixed bounded interval. 
The $h^1$  terms in $\gamma(q)$ are determined by the sheaves $\mathcal S^q( m(q) )$ just described. 
Hence they inherit the periodic behavior.  {\em End of the sketch.} \hfill $\square$

Now we return to the function $\hk{R,I}{e}$. By (\ref{seq_sheaf}), (\ref{HKsheaf}) and (\ref{HKhomology}), 
it can be expressed in terms of the Frobenius pull-back 
\begin{equation}\label{hksheaf2} 
\begin{array}{lcl} 
\hk{R,I}{e}  =  \hkq{R,I}{q} & = & \displaystyle{\sum_{m \geq 0} h^0( \mathcal O^q (m)) - 
 \sum_{i=1}^n \left ( \sum_{m\geq 0} h^0( ( \mathcal O(- d_i))^q(m) )\right) } \\ \\
 & & +  \displaystyle{ \sum_{m\geq 0} h^0 \left ( \mathcal S^q (m) \right ) } ,
 \end{array}
\end{equation}  
where $\mathcal S = Syz(f_1, \dots, f_n)$.
The last summand in (\ref{hksheaf2}) can be written using (\ref{global}).
The range of the sum is given in \cite[Lemma~9.4]{Br3lecture}. 
Hence by Theorem~\ref{geoBrenner}, the Hilbert-Kunz function has the form 
\[ \hkq{R,I}{q} = \alpha q^d + \beta q^{d-1}  + \gamma(q) . \] 
Furthermore, we observe that the coefficients of $q$ in (\ref{global}) are linear combinations of ranks and degrees
which are additive on short exact sequences. Thus they will cancel each other in the sum (\ref{hksheaf2}). 
Hence the linear coefficient $\beta$ always vanishes in the case of smooth projective curves. 
This gives the Hilbert-Kunz function the desired form $\hkq{R,I}{q} = \alpha q^2 + \gamma(q)$ in (\ref{fcn:Brenner}). 

By Theorem~\ref{geoBrenner}, we know that $\gamma(q)$ is bounded. 
In addition, if $\kappa$ is the algebraic closure of a finite field, $\gamma(q)$ is an eventually periodic function (see \cite[Theorem~6.1]{Br2}). 
We note that the proof of (\ref{fcn:Brenner}) and the periodicity of $\gamma(q)$ has been known prior to \cite{Br2} for other cases  such as in \cite{Con1995, BuC,    Mon2, Tei2002, FaT, MoT2004, Mon5, BrHe2006,  MoT2006} 
where the finiteness condition on the field is not needed. On the other hand, 
if $\kappa$ is not the algebraic closure of a finite field, {i.e.}, if $\kappa$ is transcendental over any finite subfield, then 
whether or not $\gamma(q)$ is eventually periodic is an open question in general.

A similar approach for smooth algebraic curves can be found in Trivedi~\cite{Tr2005}.  
The above sheaf theoretic approach depends heavily on the smoothness condition on the curve $Y$. 
Without smoothness, the Frobenius functor is not an exact functor,
and torsion free sheaves are not necessarily locally free. 
Monsky \cite{Mon2007} extends Brenner and Trivedi's method to irreducible projective plane curves
without smoothness, and then a linear term appears in $\hkq{R,I}{q}$. 
The following Theorem~\ref{planecurve}, quoted from \cite[Theorem~9.10]{Br3lecture}, 
summarizes the very interesting results of \cite{Mon2007} which show that
appearance of $\beta$ reflects the existence of singularities. 
In this theorem, the Hilbert-Kunz functions $\hkq{R,I}{q}$ are taken with respect to $I=(x, y,z)$ and arbitrary $f$. 
Then in \cite{Mon2011}, Monsky focused on nodal curves recovering a result of Pardue~\cite{Par1994} when $I=(x,y,z)$, 
and also extended the result to cases of arbitrary $I$ primary to $(x,y,z)$. More specifically, 
Monsky~\cite{Mon2011} applied Brenner and Trivedi's method, together with the theory for indecomposable vector bundles developed by Burban\cite{Bu2012}, and obtained a sharp result. Precisely, one has 
$\hkq{R,I}{q} = \alpha q^2 + \beta q + \gamma(q)$ for any ideal $I$ primary to $(x,y,z)$. 
The leading and second coefficients $\alpha$ and $\beta$ (nonzero) are constant with explicit formulas given in the paper. Moreover, when $p\neq 3$, $\gamma(q)$ is a periodic function depending on $q$ modulo $3$; otherwise $\gamma(q)$ is a constant. Thus $\hkq{R,I}{q}$ is a polynomial function when $p=3$. 

\begin{theorem}\cite[Theorems~I \& II]{Mon2007}\label{planecurve}
 Let $C=\proj R$ be an irreducible projective plane curve where $R= \kappa[x,y,z]/(f)$ with a homogenous $f$ of degree $d$. Then 
  \[ \hk{R}{e} = \ell(R/(x^{p^e}, y^{p^e}, z^{p^e}, f) = (\frac{3d}{4} + \frac{a^2}{4d} ) (p^e)^2 + b_e ^{\ast} \cdot \frac{a}{d} p^e + \gamma(p^e) , \]
where $a \in \mathbb Z[\frac{1}{p}]$, $0 < a < d$, and $b_e ^{\ast}$ is a periodic integer-valued function. 
The function $ b_e ^{\ast}$ can be written as $ b_e ^{\ast} = \sum_Q \beta_Q^{\ast} (e)$ 
where the sum runs over the singular points $Q$ of $C$ 
and $ \beta_Q^{\ast}(e)$ is an integer-valued periodic function for each $Q$. 
The function $\gamma(p^e)$ is bounded and is eventually period if $\kappa$ is an algebraic closure of a finite field. 
\end{theorem}

We remark that the coefficient of $p^e$ in  the above theorem does not necessarily vanish and 
in fact it is described by the singular points on $C$. 
Monsky's papers \cite{Mon2007, Mon2011} provide some  examples and are interesting resources for those who are interested in the technique presented in this subsection. 

Finally we would like to point out that this sheaf theoretic approach has also been applied to study the uniformity of 
the limit Hilbert-Kunz function as described in Application~\ref{limitHK}.

\subsection{Via Local Chern Characters}\label{sub_Chern}
In this subsection, we present an application of algebraic intersection theory to the study of  Hilbert-Kunz functions. 
We start with a brief introduction to some fundamental notions of this theory in the local ring setting. Then we present
Kurano's theorem that expresses $\hkq{R,I}{q}$ in terms of local Chern characters, followed by discussions on
the vanishing property of the second coefficient, and the existence of Hilbert-Kunz functions in a polynomial form 
with desired coefficients. 

Let $(R, \pp m)$ be a Noetherian local domain of positive characteristic $p$ and $\dim R=d$. 
We assume also that $R$ is a homomorphic image of a regular local ring (in order to define its Todd class). Let $\mathbb G_{\bullet}$ be a bounded complex of free modules of finite ranks such that all the homology modules have finite length. 
We call such $\mathbb G_{\bullet}$ a {\em perfect complex supported at $\pp m$}, although we will abbreviate this 
to {\it perfect complex} here.  

The local Chern character $\ch(\mathbb G_{\bullet} )$ is a fundamental piece of machinery in intersection theory 
and is often defined in the setting of projective schemes. 
Here we describe $\ch(\mathbb G_{\bullet} )$ and other relevant terms following Roberts~\cite{R1} where the 
definitions over local rings (in the setting of affine schemes) are provided in addition to their projective version. 
(See also Fulton~\cite{F} for the general theory.) 

We begin by defining the {\em Chow group} of $R$ which is decomposed into a direct sum of subgroups: 
$A_*(R) = \oplus_{i=0}^{d} A_i(R)$ where each $A_i(R)$ is a quotient of a free group $Z_i(R)$  modulo rational equivalence and $Z_i(R)$ is generated by prime ideals $\pp p$ of $\dim R/\pp p = i$. 
We use $[R/\pp p]$ to denote an element in $Z_i(R)$ and, by abuse of notation, 
its equivalence class in $A_i(R)$. 
For any prime ideal $\pp q$ of $\dim R/\pp q = i+1$ and $x $ not in $\pp q$, 
we define $\divs(\pp q, x)$ in $Z_i(R)$ as $\divs(\pp q, x)= \sum_{\pp p} \length(R/(\pp q+(x)) )_{\pp p} [R/\pp p]$  where the summation is taken over all prime ideals $\pp p$ of dimension $i$ containing $\pp q$ and $x$.
Notice that this is a finite sum since all the $\pp p$ in the sum are minimal prime ideals of $R/( \pp q+(x) )$ and there are only finitely many of them.
{\em Rational equivalence} is the equivalence relation on  $Z_i(R)$ induced by setting all $\divs(\pp q, x)=0$.  
Applying this definition to $R/\pp m$ in place of $R$ gives the Chow group of $R/\pp m$ as a free group of rank one. Indeed $A_*(R/\pp m) = A_0(R/\pp m) = \mathbb Z \cdot [R/\pp m]$.
Its tensor product with the rational number field is denoted $A_*(R/\pp m)_{\mathbb Q} := A_*(R/\pp m) \otimes_{\mathbb Z} \mathbb Q$.  ({\em c.f.} \cite[Section~1.1]{R1}).

Let $\mathbb G_{\bullet}$ be a perfect complex as defined above. 
The {\em local Chern character} $\ch(\mathbb G_{\bullet} )$ is a map from $A_*(R)$ to $A_*(R/\pp m)_{\mathbb Q}$. 
For any class $\alpha$ in $A_*(R)$, the notation $\ch(\mathbb G_{\bullet} )(\alpha)$ means applying sufficiently many hyperplane sections arising from the complex $\mathbb G_{\bullet}$ to the class $\alpha$. 
Many details are involved to assure that $\ch(\mathbb G_{\bullet} )$ is a well-defined map. For this we refer the readers to Roberts~\cite[Section~11.5]{R1} which will further lead to appropriate references for the precise definitions in each step along the way. Here we mention some properties about the Chern characters. 
First, the local Chern character is additive. Second, it decomposes  as 
$\ch(\mathbb G_{\bullet} ) = \ch_d(\mathbb G_{\bullet} ) + \ch_{d-1}(\mathbb G_{\bullet} ) + \cdots + \ch_0(\mathbb G_{\bullet} )$. 
Each $\ch_i(\mathbb G_{\bullet} )$ results in $i$ hyperplane sections so the dimension will drop exactly by $i$. 
Hence for any $\alpha \in A_*(R)$, if we consider its decomposition in $ A_*(R)$ as 
$\alpha = \alpha_d + \alpha_{d-1} + \cdots + \alpha_0$,  then $\ch_i(\mathbb G_{\bullet}) (\alpha_j) =0$ for $i\neq j$, 
and  thus $\ch(\mathbb G_{\bullet} )(\alpha) = \sum _{i=0}^d \ch_i(\mathbb G_{\bullet} )(\alpha_i)$. 
We note also that the local Chern character can be defined for a bounded complex supported at a larger spectrum; that is, for those bounded complexes whose homology modules need not have finite length. In this case, the target Chow group will also be larger. This fact will be used in the definition of the Todd class below.

Now we describe the Todd class. 
Following Roberts~\cite[Section~12.4]{R1}, it is defined for bounded complexes of finitely generated modules. 
Let   $\mathbb M_{\bullet}$ be such a complex supported at an ideal $\pp a$. 
By assumption, $R$ is a homomorphic image of a regular local ring, say $S$. 
Viewing $\mathbb M_{\bullet}$ as a complex over $S$, we take a finite free resolution 
$\mathbb H_{\bullet}$ of $\mathbb M_{\bullet}$ over $S$. 
Consider the class $[S]$ defined by the zero ideal in the Chow group $A_*(S)$. 
The {\em Todd class} of $\mathbb M_{\bullet}$ is given by 
$\td( \mathbb M_{\bullet} ) = \ch(\mathbb H_{\bullet})([S])$ where 
$\ch(\mathbb H_{\bullet})$ is defined in a more general sense than previously described and  
as a result, $\td( \mathbb M_{\bullet} )$ is a class in  $A_*(R/\pp a)_{\mathbb Q}$. 
It is important to note that the definition of the Todd class is independent of the choice of the regular local ring $S$. 
The Todd class can also be defined for an $R$-module $M$ by considering it as a complex concentrated in degree 0. 
In the special case of $\mathbb M_{\bullet}$ being a perfect complex, \cite[Theorem~12.4.2]{R1} relates the Todd class to 
the Euler characteristic
\begin{equation}\label{todd}
\td(\mathbb M_{\bullet}) = \chi(\mathbb M_{\bullet})[R/\pp m] , 
\end{equation}
which is part of the local Riemann-Roch formula to be described next.

Local Chern characters do not have well-behaved functorial properties, namely, they do not commute with push-forwards or pull-backs. Todd classes are introduced to fill in this gap via Riemann-Roch Theorem ({\em c.f.} Serre~\cite[Introduction]{Ser1961}). 
The Riemann-Roch Theorem can be stated in many different forms depending on the context. The local Riemann-Roch Formula (\cite[Section~12.6]{R1}) relates local Chern characters to the Euler characteristic. It states that 
\begin{equation}\label{RR}
\td( \mathbb G_{\bullet} \otimes_R M) = \ch(\mathbb G_{\bullet}) (\td(M)).
\end{equation}
By assumption, $\mathbb G_{\bullet} $ is perfect which implies that $\mathbb G_{\bullet} \otimes_R M$ is also perfect. 
Thus (\ref{todd}) and (\ref{RR}) together give 
\begin{equation}\label{EularRR}
\chi( \mathbb G_{\bullet} \otimes_R M)[R/\pp m] = \ch(\mathbb G_{\bullet}) (\td(M)). 
\end{equation}
Notice that the local Chern character $\ch(\mathbb G_{\bullet} )$ maps into the rational Chow group of $R/\pp m$, 
or precisely $\mathbb Q \cdot [R/\pp m]$. Thus its image may be identified with a rational number. 

Now replacing $M$ by $^eR$ yields
\begin{equation}\label{fcn:Euler}  
\chi_{\mathbb G_{\bullet}}( {^eR} )  =   \ch( \mathbb G_{\bullet} ) (\td ( ^eR ) ) .
\end{equation}
(See also Fulton~\cite[Example~18.3.12]{F}). 

\begin{theorem}[Kurano~\cite{K16, K5}]\label{fcn:Kurano}
Let  $R$ be a local ring that satisfies the following conditions:  \\
$({i})$ $R$ is the homomorphic image of a regular local ring whose residue field is perfect; \\
$({ii})$ $R$ is $F$-finite; \\ 
$({iii})$ $R$ is a Cohen-Macaulay ring. \\
Let $I$ be an $\pp m$-primary ideal of finite projective dimension and let $\mathbb G_{\bullet}$ be a finite free resolution of $R/I$. 
Then 
\[ 
\hk{R,I}{e} = (\ch(\mathbb G_{\bullet}) (c_d) ) (p^e)^d + (\ch (\mathbb G_{\bullet}) (c_{d-1}) ) (p^e)^{d-1} + 
 \cdots + ( \ch (\mathbb G_{\bullet}) ( c_0 ) ) , 
\]
where $c_i$ is the $i$-th Todd class of $R$ in the  $i$-th component of the Chow group, 
namely $\td( {R} ) = c_d + \cdots + c_0 \in A_*(R)$. 
\end{theorem}

With the prior preparation, we now bring up two main facts that lead to the results above:
 (1) one has  $\td_i(  {^e\!R} ) = p^{i e} c_i  \in A_i(R)$  (\cite[Lemma~2.2(iii)]{K5}); and 
(2) for any $i$, the $i$-th Chern character maps $c_i$ to a rational number but all other $c_j$ to zero.
In other words, $\ch( \mathbb G_{\bullet}) ( c_i ) = \ch_i(\mathbb G_{\bullet}) ( c_i) $ is a 
rational number in $A_*(R/\pp m)_{\mathbb Q} \cong \mathbb Q$. 
Therefore by (\ref{euler}) and (\ref{fcn:Euler})
\[
\begin{array}{lll} 
\hk{R,I}{e}  & = & \chi_{\mathbb G_{\bullet}} ( {^e\!R} ) \\
& = & \ch( \mathbb G_{\bullet} ) (\td ( {^e\!R} ) ) \\
& = & \sum_{i=0}^d \ch_i(\mathbb G_{\bullet}) \left ( (p^e)^d c_d  + (p^e)^{d-1}   c_{d-1} + \cdots + c_0 \right) \\ 
& = & \sum_{i=0}^d (\ch_i(\mathbb G_{\bullet}) (c_i ) ) (p^e)^i . 
\end{array} 
\]
Or equivalently with $q=p^e$,
\begin{equation}\label{ChernHK}
 \hkq{R,I}{q} = (\ch(\mathbb G_{\bullet}) ( c_d ) q^d + (\ch (\mathbb G_{\bullet}) ( c_{d-1}) ) q^{d-1} + 
 \cdots + ( \ch (\mathbb G_{\bullet}) ( c_0 ) ). 
\end{equation} 
This is a polynomial in $q$ of degree $d$ and all the coefficients are rational numbers. 

The leading coefficient, which never vanishes, is the Hilbert-Kunz multiplicity. However, 
the second coefficient $\beta_I(R) = \ch_{d-1}(\mathbb G_{\bullet}) (c_{d-1} ) $ can sometimes vanish, 
for instance, in the case of a two-dimensional normal domain as observed in \cite{HMM} and \cite{Br2}.  
More generally, from Kurano~\cite[Corollary~1.4]{K5}, 
we learn that if $R$ satisfies the conditions {\em $($i$)$} and {\em $($ii$)$} above, and $R$ is normal and $\mathbb Q$-Gorenstein ({i.e.}, the canonical module defines a torsion element in the divisor class group), then $\beta_I(R)$ vanishes. This is also true for the second coefficient in 
the function $\hkq{M,I}{q}$ of a  module $M$. 

In \cite[Theorem~3.5]{ChK1}, the vanishing property of $\beta$ is characterized by the classes in the Chow group. 
For example, the vanishing of $\beta_I(R)$ for every $I$ is equivalent to the fact that the second top Todd class 
$\tau_{d-1}=\td_{d-1}( R)$ is {\em numerically equivalent} to zero, i.e., 
$\ch(\mathbb G_{\bullet}) (\tau_{d-1})=0$ for any perfect complex $\mathbb G_{\bullet}$.
Furthermore, if  the localization $R_{\pp p}$ is Gorenstein for all minimal prime ideals $\pp p$, then $\beta_I(R) =0$
if and only if $\tau_{d-1}$ and its canonical module $\omega_R$ are numerically equivalent, namely, 
 $\ch(\mathbb G_{\bullet} )(\tau_{d-1}) = \ch(\mathbb G_{\bullet}) ( \omega_R)$ for any $\mathbb G_{\bullet}$
 that is again perfect. 

Theorem~\ref{fcn:Kurano} also makes it possible to prove the existence of Cohen-Macaulay local rings 
such that the Hilbert-Kunz functions have polynomial expressions as in (\ref{ChernHK}) and their rational coefficients have the desired positive, negative or vanishing properties (\cite[Theorem~1.1]{ChK2}). 
The theorem states that if $\epsilon_0, \epsilon_1, \dots, \epsilon_d$ are integers such that 
$\epsilon_i=0$ for $i \leq d/2$, $\epsilon_i= -1, 0, \text{ or}, 1$ for $d/2< i <d$ and $\epsilon_d=1$, then there exists a $d$-dimensional Cohen-Macaulay local ring $R$ of characteristic $p$, an $\pp m$-primary ideal $I$ of $R$ of finite projective dimension, and positive rational numbers $ \beta_0, \beta_1, \dots, \beta_d$ such that 
$ \hkq{R,I}{q} = \sum_{i=0}^{d} \epsilon_i \beta_i q^i $ for all $i >0$. 
In \cite{ChK2}, a convex cone in a finite dimensional vector space, named Cohen-Macaulay cone, spanned by maximal Cohen-Macaulay modules is introduced to carry out the proof. This theorem proves the existence but does not offer a constructive method to build a ring with the desired Hilbert-Kunz function.

In general, it is difficult to construct rings that have specific forms of Hilbert-Kunz functions. 
This is true even in the setting of affine semigroup rings. (See also Remark~\ref{construction}.) 
Investigations in this direction can provide not only desired Hilbert-Kunz functions but also, as shown in Theorem~\ref{fcn:Kurano}, a possible approach to access local 
Chern characters whose values are equally, if not more, challenging to obtain.

\subsection{Via Bruns-Gubeladze (BG) Decomposition}\label{sub_BG}
This subsection describes a cellular decomposition on $\mathbb R^d$ and its fundamental domain 
when $R$ is an affine semigroup ring.
There are one-to-one correspondences between the set of full dimensional cells in this decomposition, 
the set of conic divisor classes, and the set of rank one modules as direct summands of an extension ring of $R$. 
Then we present Bruns's ideas of using these correspondences to calculate Hilbert-Kunz functions. 

Let $\kappa$ be an algebraically closed field of positive characteristic $p$ and $R$ be a $d$-dimensional normal $\kappa$-subalgebra of 
the polynomial ring $\kappa[t_1, \cdots, t_d]$ generated by finitely many monomials.
Then the exponents of monomials in $R$ form a finitely generated monoid $M$ in $\mathbb Z^d$.
Such an $R$ is called an {\em affine semigroup ring}, denoted by $R=\kappa[M]$.  We use $\mathbb Q_{\geq 0}M$ 
(and $\mathbb R_{\geq 0}M$) to denote the extended rational cone spanned by $M$ in $\mathbb Q^d$ ({\em resp.} $\mathbb R^d$), 
and $\mathbb ZM$ to denote the free abelian group generated by $M$. 
Then $\mathbb Z M$ has rank $d$ since $\dim R=d$.
We assume $\mathbb Z M$ equals the ambient group $\mathbb Z^d$.
The assumption that $R$ is a normal domain is equivalent to a condition on the semigroup:
$\mathbb ZM \cap \mathbb Q_{\geq 0}M = M$ ({\em c.f.} \cite{Ho1972}). 
Let $\frac{1}{n} M = \{ \frac{x}{n}: x\in M\}$ and consider $R^{\frac{1}{n}}= \kappa[ \frac{1}{n} M ]$ as an extension ring of $R$. 
Section~\ref{affineEhrhart} is devoted to the affine semigroup rings 
where relevant definitions  and terms will be stated in more detail. 

The idea of a cellular decomposition of $\mathbb R^d/\mathbb Z^d$ as a quotient group grew out of 
Bruns and Gubeladze's work \cite{BruG} that investigates the minimal number of generators and depth of the divisorial ideal classes of a normal affine semigroup ring in general. 
Later it was constructed precisely with the lattice structure in \cite{Bru1} which we now refer to as {\em BG decomposition}. 
Some readers might find it helpful to read Section~\ref{affineEhrhart} before proceeding to the current subsection. 
Here as part of the section for techniques, we present the main theorems on the BG decomposition in \cite{Br1} that leads to the computation of the Hilbert-Kunz function of an affine semigroup ring $R$ with respect to the maximal ideal generated by all monomials other than 1. 

Let $\sigma_1, \dots, \sigma_{\ell}$ be the linear functionals that define the support hyperplanes for $M$ so that 
 $\sigma_i(M) \geq 0$ for all $i=1, \dots, \ell$.  Since $\rank \gp M = \dim \kappa[M]=d$, we have $\ell \geq d$ and 
the set of support hyperplanes is uniquely determined if it is irredundant.  
The map $\sigma=( \sigma_1, \dots, \sigma_{\ell}) : M \rightarrow \mathbb Z^{\ell}$ 
is called the {\em standard embedding} ({\em c.f.} \cite{BruG}) which has also been used in Hochster~\cite{Ho1972} and Stanley~\cite{St3}. Such a $\sigma$ transforms a normal semigroup to an isomorphic subsemigroup in $\mathbb Z^{\ell}$ that is said to be {\em full} in \cite{Ho1972} or {\em pure} in \cite{BruG}. 
While normality is independent of the embeddings, ``fullness" ({\em resp.} ``pureness") is a notion relative to the embedding. 
But we will not get into the details on this issue here.   

Obviously $\sigma$ can be extended to define a linear homomorphism on $\mathbb R^d$ and $\sigma(\mathbb R^d) \cong \mathbb R^d \subseteq \mathbb R^{\ell}$. 
We obtain a cell decomposition $\Gamma$ of $\mathbb R^d$ by the hyperplanes
\[ H_{i,z} = \{ x\in \mathbb R^d | \sigma_i(x) =z\}, \,\,\, i =1, \dots, \ell, \text{ and } z \in \mathbb Z , \] 
which then induces a cell decomposition $\bar \Gamma$ on $\mathbb R^d/\mathbb Z^d$. 
For concreteness, both $\Gamma$ and $\bar{\Gamma}$ will both be called the BG decomposition.

A {\em conic divisorial} ideal in \cite{BruG, Bru1} is defined to be 
\[ C(y) = \kappa\cdot ( \mathbb Z^d \cap (y + \mathbb R_{+} M)), \,\,\, y \in \mathbb R^d , \]
and it should be considered as the ideal of $R$ generated by the monomials corresponding to 
the lattice points on the right hand side. 
Conic divisorial ideals are among the divisorial ideals that generate the divisor class group $\cl(R)$. 
The set of conic divisor classes in $\cl(R)$ contains all torsion elements but it can be larger if $\cl(R)$ is nontorsion 
(\cite[p.~141]{BruG}, \cite[Corollary1.3]{Bru1}).
The following bijections are the key to the discussions in this subsection. 

\begin{theorem}\cite[Corollary~1.2]{Bru1}\label{thm_BG}
The following sets are in a bijective correspondence: \\
$(a)$ the set of conic divisor classes\,$;$ \\
$(b)$ the fibers of the map $\mathbb R^d/\mathbb Z^d \rightarrow \mathbb R^{\ell} / \sigma(\mathbb Z^d)$ 
 induced by $x \mapsto \lceil \sigma( x) \rceil ;$ \\
 $(c)$ the full dimensional cells of $\bar \Gamma$.
\end{theorem}

Let $\bar \gamma$ be a full-dimensional cell in $\bar \Gamma$ represented by some full-dimensional 
$\gamma$ in $\Gamma$. 
Fix an element $y$ in $\gamma$, the {\em upper closure} of $\gamma$ is defined to be 
\[ \lceil \gamma \rceil = \{ x \in \mathbb R^d | \lceil \sigma(x)\rceil = \lceil \sigma (y) \rceil \} . \]
This closure, however, is independent of $y$. 
Theorem~\ref{thm_BG} shows that all the conic ideals $\mathcal C(x)$ with $x\in \lceil \gamma \rceil$ coincide in $\cl(R)$. 
Thus we may use $\mathcal C_{\gamma}$ to denote such a conic ideal class. 
We recall $R^{\frac{1}{n}} =  k[  \frac{1}{n}  M ]$ for any positive integer $n$.
For each residue class $c \in ( \frac{1}{n} \gp M) / \gp M$, let $I_c$ be the ideal generated by $R^{\frac{1}{n}} \cap k \cdot c$. 
Since $M$ is normal, $I_c$ is equivalently generated by monomials corresponding to 
$ \mathbb R_{\geq 0} M \cap (c+ \mathbb Z_{\geq 0}M ) $, and  so it defines the same conic divisor class as $\mathcal C(-c)$ which must be isomorphic to $\mathcal C_{\gamma}$ for some full dimensional cell $\gamma$ by Theorem~\ref{thm_BG}.
This shows that the set of $I_c$ corresponds to a subset of  the conic divisor classes. 
Then by \cite[Proposition~3.6]{BruG} we know that the set of $I_c$ and the set of the conic divisor classes are in a  
one-to-one correspondence. 

By \cite[Theorem~3.2(a)]{BruG}, 
$R^{\frac{1}{n}}$ can be decomposed as an $R$-module into the direct sum of 
$I_c$ over all  $c\in (\frac{1}{n}\mathbb ZM)/\mathbb ZM$ with $\rank_R I_c =1$. 
This implies that the rank of $R^{\frac{1}{n}}$ is the number of the residue classes in 
$(\frac{1}{n} \mathbb Z M)/ \mathbb Z M$. 
Furthermore, each $I_c$ corresponds to a unique conic divisor class $\mathcal C_{\gamma}$, 
but there may be multiple $I_c$'s that correspond to the same $\mathcal C_{\gamma}$. 
Let $\nu_{\gamma}(n)$ be the multiplicity with which the isomorphism class of $\mathcal C_{\gamma}$ occurs in the decomposition of $R^{\frac{1}{n}}$ as an $R$-module. From the previous discussion, it is not difficult to see  that 
$\nu_{ \gamma} (n)=  \#\{ I_c | c \in \lceil \gamma \rceil \}$. 

\begin{theorem}\label{thm_nu}\cite[Theorem~3.1]{Bru1}
Let $\gamma$ be a full-dimensional cell in $\Gamma$. Then
\[ \nu_{\gamma} (n)= \#( \lceil \gamma \rceil  \cap \frac{1}{n}\mathbb Z^d). \] 
Furthermore, there exists a quasi-polynomial $q_{\gamma}(n): \mathbb Z \rightarrow \mathbb Z$ 
with rational coefficients such that 
$q_{\gamma}(n) = \nu_{\gamma}(n)$ for all $n \geq 1$. In particular, 
\[ q_{\gamma}(n) = \operatorname{vol}( \gamma ) \, n^d + b_{\gamma} n^{d-1} + \tilde{q}_{\gamma}(n), 
\,\,\, n \in \mathbb Z , \]
where $\operatorname{vol}( \gamma )$ is the volume of $\gamma$,  $b_{\gamma}$ is  a constant 
and  $\tilde{q}_{\gamma}(n)$ is a quasi-polynomial of degree $d-2$. 
\end{theorem}

The first statement in Theorem~\ref{thm_nu} holds due to the fact that $I_c \cong C(-c)$ if and only if 
$(-c + \mathbb Z^d) \cap \lceil \gamma \rceil \neq \emptyset$.  
Given this, the relation of $\nu_{\gamma}(n)$ to a quasipolynomial $q_{\gamma}(n)$ is not a surprise.
In fact, since all the vertices of $\lceil \gamma \rceil$ have rational coordinates, we observe that 
\[\#( \lceil \gamma \rceil  \cap \frac{1}{n}\mathbb Z^d) = 
   \#( (n \cdot \lceil \gamma \rceil)  \cap \mathbb Z^d) , \]
where $  (n \cdot \lceil \gamma \rceil)$ is the region in $\mathbb R^d$ whose vertices are $n$ times those of 
$\lceil \gamma \rceil$. Therefore, $\nu_{\gamma}(n)$ counts the number of lattice points in the dilated 
$\lceil \gamma \rceil$. This is exactly the generalized Ehrhart function which will be described in Section~\ref{affineEhrhart}. 
And that $\nu_{\gamma}(n)$ is a quasipolynomial is a direct consequence of  the a polycell version of Theorem~\ref{genEhrhart} which will be explained also in the next section. It has been known that the leading coefficient of the generalized Ehrhart polynomial can be described by the volume and hence it is a constant. The remaining coefficients 
are periodic functions in $n$. However, in the setting of the previous Theorem~\ref{thm_nu}, Bruns further refines 
the results from Theorem~\ref{genEhrhart} and identifies that the second coefficient is also a constant. This refinement 
leads to the existence of the second coefficient of the Hilbert-Kunz function in the following Theorem~\ref{BGhk} that matches the more general case of normal domains proved in \cite{HMM} as discussed in Subsection~\ref{sub_classgp}. 

The BG decomposition takes place for any positive integer $n$. However, if the coefficient field $\kappa$ has positive characteristic $p$, we may restrict $n$ to be powers of $p$ to obtain Hilbert-Kunz functions. In what follows, we use 
$\mu_R(\cdot)$ to denote the minimal number of generators of an $R$-module. 

\begin{theorem}\cite[Corollary~3.2(a)]{Bru1}\label{BGhk}
Let $\kappa$ be an algebraically closed field of characteristic $p >0$ and $M$ a normal finitely generated monoid of rank $d$. Set $R=\kappa[M]$ and let $\pp m$ be the maximal ideal of $R$ generated by the monomials whose exponents are in 
$M \backslash \{0\}$. 
Let $\gamma$ be a 
full-dimensional cell as described above and $C_{\gamma}$ be the conic ideal corresponding to $\gamma$. Then the Hilbert-Kunz function of $R$ with respect to $\pp m$ is a quasipolynomial with constant leading and second coefficients. Precisely 
\begin{equation}\label{eq_BGhk}
 \hk{R}{e} = \displaystyle{\sum_{\gamma}} \,\, \mu_R(C_{\gamma}) \,\nu_{\gamma}(p^e) 
\end{equation}
for all positive integers $e$. 
\end{theorem} 
\begin{proof} 
Recall $q=p^e$ and $R^{\frac{1}{q}} = \kappa[\frac{1}{q}M]$.  Since $\kappa$ is perfect,  we have the following 
$\kappa$-algebra isomorphism induced by the Frobenius map,
\[ R/ \pp m^{[q]} \cong R^{\frac{1}{q}}/ \pp m R^{\frac{1}{q}}. \] 
Obviously $\dim_{\kappa} R/ \pp m^{[q]} = \dim_{\kappa} R^{\frac{1}{q}}/ \pp m R^{\frac{1}{q}} 
  \leq \mu_R(R^{\frac{1}{q}}) $. 
By Nakayama's Lemma, the reverse of the last inequality also holds. Hence we have 
$ \dim_{\kappa} R^{\frac{1}{q}}/ \pp m R^{\frac{1}{q}} =  \mu_R(R^{\frac{1}{q}})$.

We recall also that $I_c$'s are the rank one direct summands in the decomposition of $R^{\frac{1}{q}}$ as $R$-modules. 
Therefore, we have 
\[ \begin{array}{lcl}
 \dim_{\kappa} R/ \pp m^{[q]} 
   =  \mu_R(R^{\frac{1}{q}})  
  & = & \mu_R( \oplus_c I_c )\\
  & = &  \sum_c \mu_R(C_{\gamma}), \,\,\, \text{ (since $I_c \cong C_{\gamma}$ for some $\gamma$) } \\
  & = &  \sum_{\bar{\gamma}}   \mu_R(C_{\gamma}) \cdot \#(\text{$I_c$ isomorphic to $C_{\gamma}$}) \\
  & = &  \sum_{\bar{\gamma}}   \mu_R(C_{\gamma})  \, \nu_{\gamma}(q) .
  \end{array} \]
Hence 
\[ \hkq{R}{q} = \ell_R(R/ \pp m^{[q]})  =   \dim_{\kappa} R/ \pp m^{[q]}  
   =  \sum_{\bar{\gamma}}   \mu_R(C_{\gamma})  \, \nu_{\gamma}(q) .  \]
Moreover since $\mu_R(C_{\gamma})$ is a constant independent from $q$, it is clear that $\hkq{R}{q}$ has all the properties enjoyed by $\nu_{\gamma}(q)$ in Theorem~\ref{thm_nu}, 
namely, $ \hkq{R}{q}$ is a quasipolynomial and that its leading and the second coefficients are both constants.
\end{proof} 

We will return to (\ref{eq_BGhk}) after the discussion of Ehrhart Theory in Section~\ref{affineEhrhart} (see Remark~\ref{rmk_BG}).

The next example describes the BG decomposition of the semigroup generating the rational normal cone of degree 2. 
This is a special case of Example~\ref{regcone} with $g=2$. 

\begin{example}\label{ex_BG} 
Let $R=k[s, st, st^2]$ where $M$ is the semigroup in $\mathbb Z^2$ generated by $(1,0)$, $ (1,1)$, $(1,2)$. 
Then $\sigma$ is the linear map $\sigma=(d_2, 2d_1-d_2): M \rightarrow \mathbb Z^2$. 
Note that in this case, $M$ generates a simplicial cone in $\mathbb R^2$. 
The slanted grids in Figure~1 below show the decomposition $\Gamma$ of $\mathbb R^2$. 
Full dimensional cells that are equivalent in $\overline \Gamma$ are labeled by $\bigcirc$ and $\triangle$ respectively just to show a few examples. Note that there are two distinct classes. We also show the graph of a single open cell $\gamma$ and its closure $\lceil \gamma \rceil$. 

\medskip

\begin{center} 
  \begin{tabular}{ccccc}
  \begin{minipage}{1.5in} 
       \begin{center}      \epsfig{file=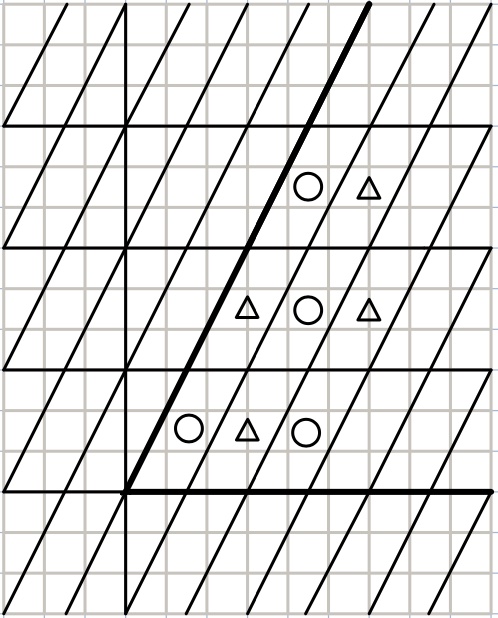, height=1.5in}  \end{center} 
  \end{minipage} &  \hspace{.2in} & 
     \epsfig{file=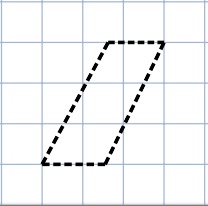, height=.4in}   
     & \hspace{.1in} &
      \epsfig{file=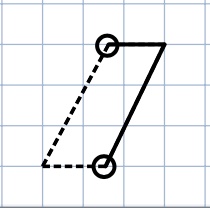, height=.4in}  \\
   {\em {\small $\Gamma$ and $\overline \Gamma$}} &   \hspace{.1in} & {\em {\small One open cell $\gamma$ }}  & \hspace{.1in} &  { \em {\small Closure $\lceil \gamma \rceil $ }}
    \end{tabular} \\  
   \vspace{3mm}
{\small {\em Figure}~1: $\Gamma$ and the equivalence classes in $\overline \Gamma$.}
\end{center}

Let $\mathcal P = \{ x \in \mathbb R_{\geq 0}M | x \notin {\bf d} + \mathbb R_{\geq 0} M \text{ for any nonzero } {\bf d} \in M\} $.
Then $\mathcal P$ can be tessellated by finitely many full-dimensional cells. 
Each non-equivalent cell may occur in a different number of copies. (See {\em Figure}~2.) 
For some $\gamma$, $P$ may contain its entire closure  but not always. 
Such decomposition is unique in the sense that each full dimensional cell in $\bar \Gamma$ is 
 identified with a unique conic ideal class. 
The closure of these  semi-open  rational polycells $\lceil \gamma \rceil $ are convex.

\begin{center} 
 \begin{tabular}{c} 
    \epsfig{file=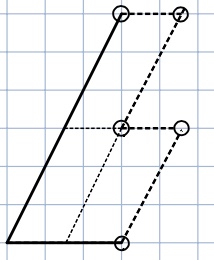, height=.8in}   \\ 
    {\small {\em Figure} 2: $\mathcal P$ by $\overline \Gamma$ }
 \end{tabular}    
\end{center}

We write $\gamma_1 $ for a cell labeled by $\bigcirc$ and $\gamma_2$ for one labeled by $\triangle$.
By Theorem~\ref{BGhk}, 
\[ \hk{R}{e} = \mu_R(C_{ \gamma_1 }) \,\nu_{\gamma_1}(p^e)  +  \mu_R(C_{ \gamma_2 }) \,\nu_{\gamma_2}(p^e) . \]
Next we determine the values for 
$\mu_R(C_{ \gamma_1 })$ and $\mu_R(C_{ \gamma_2 })$ which are independent from $e$. 
In fact, since $C_{\gamma}$ is a class defined by $I_c$ for some $c \in \lceil \gamma \rceil$, 
to obtain the value for $\mu_R(C_{\gamma})$, 
it suffices to consider the minimal number of generators of a divisorial ideal $I_c$.
Let $c=(-\frac{2}{3}, -\frac{2}{3}) \in \lceil \gamma_1 \rceil$, then $I_c \cong C(-c) \cong R\cdot st$ and 
$\mu_R(C_{\gamma_1}) =1$. 
And if $c= (0, - \frac{1}{3}) \in \lceil \gamma_2 \rceil$, then $I_c \cong C(-c) \cong R\cdot st + R\cdot st^2 $ and 
$\mu_R(C_{\gamma_1}) =2$.
\end{example}
 
In Section~\ref{affineEhrhart}, we will see that the Hilbert-Kunz function of an affine semigroup ring $R$ is equivalent to an Ehrhart function
which counts the number of lattice points in a certain dilated rational {\em polycell} $\mathcal P$ arising from $R$ 
which is nonconvex and not necessarily closed. 
We will discuss $\hk{R}{e}$ from this aspect there. On the other hand, when it comes to counting lattice points in a polycell (or a polytyope), there is no previously known canonical method about dissecting the polycell ({\em resp.} polytope) before proceeding to count. Via conic divisor classes, Bruns~\cite{Bru1} decomposes $\mathcal P$ into finitely many cells and each appears finitely many times. The BG decomposition is unique for any given affine semigroup. 
The periodic behavior of the Hilbert-Kunz function is mainly due to the fact that the vertices of full-dimensional open cells $\gamma$  have rational coordinates. 
Even though the polycell $\mathcal P$ is not convex, the full-dimensional cells in the BG decomposition are always convex. These facts can be very useful.

In order to fully express $\hk{R}{e}$ using BG decompositions,   it is equally important to know 
$\mu(C_ \gamma)$ for each $\overline \gamma$ in $\bar{\Gamma}$ (see Remark~\ref{twoHK} and Question~\ref{computation}). 
The significance of $\mu( C_\gamma)$ can also be seen in \cite{BruG} in which Bruns and Gubeladze use it to measure the number of Cohen-Macaulay divisor classes and prove that there are only finitely many such divisor classes.  

Finally we make a remark on the divisor class group. 
\begin{remark}\label{rmk:divisor}
We have seen that the divisor class group appears often in these techniques. First it is used in proving the existence of $\beta$ in Subsection~\ref{sub_classgp}. Then we see that the representation of the divisor classes play the most crucial role in the discussion of the periodic behaviors in the current subsection. 
The divisor classes also occur implicitly in Subsection~\ref{sub_sheaf}, because the divisor class group is generated  by the twists of the structure sheaves which are the main object dealt with in Theorem~\ref{geoBrenner}.  
So perhaps some level of finiteness condition on the divisor class group 
is responsible for the (periodic) behavior of the function. This remains to be investigated. 
\end{remark}

\subsection{Via Combinatorics}\label{sub_comb} 

A lot of attention has been given to quotients of polynomial rings by monomial or binomial ideals. 
In this subsection, we give a very brief account in this direction without developing their details. 
These can be traced back to Conca~\cite{Co1996} who considered generalized Hilbert-Kunz functions and utilized Gr\"obner basis to calculate the multiplicity of binomial hypersurfaces. Watanabe~\cite{Wat}, followed by Eto~\cite{Et}, 
approximated the multiplicity by the volume of the relevant polytope. 
(Our discussions on affine semigroup ring settings were initially inspired by \cite{Wat}.)
Eto and Yoshida~\cite{EtY} expressed the multiplicity in terms of the Stirling numbers.

Combining the techniques of Segre products as done in \cite{EtY} and Gr\"obner bases, Miller and Swanson~\cite{MiS2013} compute the Hilbert-Kunz function of the rings of the form of 
$R = \kappa [ X ] / I_2( X) $ where $X$ is an $m\times n$ matrix of indeterminates, 
$\kappa[X]$ is the polynomial ring joining all indeterminates in $X$,
and  $I_2(X)$ is the ideal generated by all the $2 \times 2$ minors of $X$. 
In \cite{MiS2013}, it is proved that the Hilbert-Kunz functions, for arbitrary $m$ and $n$, are true polynomials 
and are expressed in a recursive manner, but a closed form is provided only for the special case $m=2$. 
This study is extended by Robinson and Swanson~\cite{RoSw2015} who give explicit closed polynomial forms 
for all positive integers $m$ and $n$.

Before ending this subsection, we mention an interesting article by Batsukh and Brenner~\cite{BaB2017} in which
the notion of {\em binoid} is introduced as a commutative monoid with an absorbing element $\infty$.  
For rings of combinatorial natures such as Stanley-Reisner and toric rings and those just described above,
several questions arise. These include the rationality of multiplicity, the interpretation of the notions in the case of 
characteristic 0, and the dependence of the results on the characteristic. 
Batsuhk and Brenner propose a unifying method evolved around binoids to answer all these questions 
 simultaneously. 
 It will be interesting to further investigate if this new approach may be applied to obtain or approximate Hilbert-Kunz functions.

\section{Normal Affine Semigroup Rings and Ehrhart Theory}~\label{affineEhrhart}

In this section, we describe the generalized Ehrhart function and relate it to the Hilbert-Kunz function of an affine semigroup ring. 
Below, we give a precise definition of affine semigroup rings focusing on identifying monoid elements with
monomials of a Laurent polynomial ring. 
In this way, affine semigroup rings under consideration in Subsection~\ref{sub_BG} are naturally subrings of polynomial rings. 
Furthermore, it can be understood as the coordinate ring of an affine toric variety. 

A {\em monoid} in $\mathbb Z^d$ is a semigroup that contains an identity element. Let $M$ denote a monoid that can be generated by finitely many elements in $\mathbb Z^d$. We call such $M$ an affine semigroup.
An affine semigroup $M$ is {\em positive} if $M \cap - M = 0$; namely, the only invertible element in $M$ is the identity element. 
Let  $ \mathbb Q_{\geq 0} M$ denote the rational cone generated by $M$.  Saying an affine semigroup is  positive is equivalent to saying that its rational cone is {\em strongly convex} or {\em pointed}. 
For the discussions within this paper, an {\em affine semigroup} always stands for a monoid that is finitely generated and positive.
Let $\gp M$ denote the subgroup of $\mathbb Z^d$ generated by $M$. 
Obviously $M\subseteq \gp M \cap \mathbb Q_{\geq 0}M$. 
An affine semigroup $M$ is {\em normal} if equality holds, that is, for any $h \in \gp M$, 
we have $h \in M$ whenever $n h \in M$ for some positive integer $n$. 
Pictorially, $M$ is normal means that there is no ``hole" when 
comparing $M$ against $\gp M \cap \mathbb Q_{\geq 0} M$. 
We will assume that $M$ is normal in this section although some of the results might be true in more general settings. 

Let $\kappa$ be an algebraically closed field of any characteristic.
We write $\kappa[M]$ for the {\em affine semigroup ring} generated by $M$ over $\kappa$.  
By identifying $(i_1, \dots, i_d) \in \mathbb Z^d$ with $t_1^{i_1} \cdots  t_d^{i_d}$, 
we regard $\kappa[M]$ as a $\kappa$-subalgebra in the Laurent polynomial ring 
$S=\kappa[t_1, \dots, t_d, t_1^{-1}, \dots,  t_d^{-1} ]$.
For the purpose of our discussion, it is more convenient to work with affine semigroup rings under such
identification. However, for notational simplicity, we still write $\kappa[M]$ 
in place  of the formal $\kappa[\mathbf t^{M}]$.
It is known that $\kappa[M]$ is a normal ring if and only if the semigroup $M$ is normal (\cite{Ho1972}). 
Without loss of generality, we also assume  that $\gp M$ has rank $d$ (otherwise, replace $d$ by the rank of $\gp M$),
and that $\gp M$ equals the ambient group $\mathbb Z^d$.
Thus $\dim \kappa[M] =d$. 

Let $R=\kappa[M]$ be an affine semigroup ring as described above. Then $R$ can be viewed as a $k$-algebra generated by finitely many monomials in a Laurent polynomial ring $S$. Clearly $R$ is an integral domain. Sometimes $R$ is also called a {\em toric ring}, since it is  the coordinate ring of an affine toric variety that is constructed from lattice points in $M$~({\em c.f.} \cite{CLS,F2,MSt}). 
For instance, a convex polyhedral cone $\sigma \subset (\mathbb R^d)^{\vee}$  
defines an affine toric variety $X=\spec R$ where $R=\kappa[\sigma^{\vee} \cap \mathbb Z^d]$.
In this convention, $\sigma^{\vee} = \mathbb R_{\geq 0} M$ is an extended cone of the rational cone
$\mathbb Q_{\geq 0} M$. It can be proved that $R$ is isomorphic to the quotient of a polynomial modulo a prime ideal generated by binomials. 
Conversely a prime ideal generated by binomials always defines a toric ring and thus a semigroup ring 
({\em c.f.} \cite[Chapter~1]{CLS}).

Affine semigroup rings have all the above classical interpretations making them an interesting family of rings to study for any topic in both algebra and geometry. Next, after a brief introduction of Ehrhart theory, 
we will see how the polytopes and their enclosed lattice points interact with ideals of finite colength. 

An Ehrhart function concerns the number of lattice points contained in a dilated polytope. 
A {\em polytope}, denoted by $\Delta$, is the convex hull of finitely many points in $\mathbb R^N$. 
Precisely, for a given finite set of points $\{ {\bf v}_0, \dots, {\bf v}_g \} \subset \mathbb R^N$, 
the convex hull of ${\bf v}_0, \dots, {\bf v}_g $ is given by
$\Delta = \{ \lambda_0 {\bf v}_0 + \cdots + \lambda_g {\bf v}_r | 0 \leq \lambda_i \leq 1, \lambda_0 + \cdots \lambda_g =1\} \subset \mathbb R^N$. 
The dimension of $\Delta$ is equal to 
$\dim_{\mathbb R} \rm{Span}\{{\bf v}_1- {\bf v}_0, \dots, {\bf v}_g - {\bf v}_0 \}$. 
A $d$-polytope refers to a $d$-dimensional polytope. 
And $ n \Delta$ denotes the convex hull of $\{ n {\bf v}_0, \dots, n{\bf v}_g \}$; that is also a dilation of $\Delta$. 
The {\em lattice-point enumerator} ({\em c.f.} \cite[p.27]{BeR}) for the $n$-th dilation of $\Delta$ is defined to be 
\[ E_{\Delta}(n) = \# \left ( n \Delta \cap \mathbb Z^N \right ) . \]
We say that $\Delta$ is an {\em integral polytope} if all its vertices have integer coordinates and it is a {\em rational polytope} if 
its vertices have rational coordinates. The following theorem can be found in \cite[Theorems~3.8 and 5.6]{BeR}.

\begin{theorem}[Ehrhart~\cite{Ehr}]\label{Ehrhart} Let $\Delta$ be an integral $d$-polytope in $\mathbb Z^N$
and $n$ be a positive integer. 
\begin{enumerate} 
\item[(a)] $E_{\Delta}(n)$ is a  polynomial function in $n$ of degree $d$.
\item[(b)] If we write $E_{\Delta}(n) = a_d n^d + a_{d-1} n^{d-1} + \cdots + a_0$, then 
\[ a_d = \mbox{ the volume of the polytope $\Delta$, }  \]
\[ a_{d-1} = \frac{1}{2}  \cdot E_{( \partial(n\cdot \Delta) \cap \mathbb Z^N )} (n),  \] 
\[ a_0 = \mbox{ the Euler characteristic of $\Delta$},  \]
where $\partial$  indicates the boundary of the polytope.
\end{enumerate} 
\end{theorem} 

Whether or not other coefficients carry any significant information is unknown.

Ehrhart's Theorem can be generalized to rational polytopes. A function $f(n)$ on $\mathbb Z$ or $\mathbb Z_{\geq 0 }$ is called a {\em quasipolynomial}  of period $\tilde r$ if it can be written in polynomial form in $n$ whose coefficients are periodic functions in $n$ with $\tilde r$ as the least common multiple of their periodic lengths. Equivalently, there exist $\tilde r$ polynomials $f_1, \dots, f_{\tilde r}$ such that $f(n) = f_i(n)$ if 
$n \equiv i \pmod{\tilde r}$.

\begin{theorem}[Generalized Ehrhart's Theorem~\cite{Ehr}]\label{genEhrhart} 
Assume that $\Delta$ is a convex rational $d$-polytope. Then $E_{\Delta}(n)$ is a quasipolynomial in $n$ of degree $d$. Moreover, the period divides the least common multiple of the denominators of the coordinates of the vertices of $\Delta$. 
\end{theorem} 

Naturally, $ E_{\Delta}(n)$ is called the {\em Ehrhart polynomial} or {\em quasipolynomial} according to 
whether $\Delta$ is an integral or rational polytope. 

The proof of Theorem~\ref{genEhrhart} follows very similarly to that of Theorem~\ref{Ehrhart}, but as pointed out in Beck and Robin~\cite{BeR}, ``{\em  the arithmetic structures of Ehrhart quasipolynomials is much more subtle and less well known than that of Ehrhart polynomials}". A detailed proof following the instructions of \cite[Section~3.7]{BeR} is documented in Barco~\cite{Ba}. 
(See Sam~\cite{Sam2009} for an alternative approach.)

Note that for any polytope $\Delta$, all the faces in its boundary $\partial \Delta$ must be convex.
Let $\Delta^{\circ}=\Delta - \partial \Delta$ be the interior of $\Delta$. 
By the inclusion and exclusion principle and applying Theorem~\ref{Ehrhart} repeatedly, 
one can prove that Ehrhart's Theorem holds for $\Delta^{\circ}$ with $E_{\Delta^{\circ}}(n)$ defined analogously, 
and also for non-convex polytopes and their interiors as well. 
Similarly, Theorem~\ref{genEhrhart} which concerns rational polytopes holds for the boundary $\partial \Delta$, 
interior $\Delta^{\circ}$ and their non-covex counterparts ({\em c.f.} \cite[Section~4.6.2]{St2}).

For the following discussion, we define a {\em polycell} to be a union of finitely many polytopes, but not necessarily containing 
all the faces on the boundary of the union. We also require that if any two polytopes intersect, 
they do so only at their faces. Thus, a polycell is not necessarily convex; neither is it closed nor open. 
We use the term a {\em semi-open polycell} to emphasize that it contains only part of its boundary. 
The Ehrhart Theorem and its generalized version both hold for the semi-open polycells by the inclusion and exclusion principle.

Next we relate the counting of lattice points to $\hk{R}{e}$. We begin by doing so to Hilbert-Samuel functions. 
We use $\angle$ as a shorthand for the cone  $\mathbb R_{\geq 0} M$ in $\mathbb R^d$.
Let $\pp a$ be an ideal generated by monomials in $R=\kappa[M]$ such that $R/ \pp a$ has finite length.  
The Hilbert-Samuel function $\HS(n)= \ell (R/ \pp a^n)$ can be obtained by 
counting the lattice points in the complement of the union of the shifted cones 
$L_n = \bigcup (\bf d + \angle)$, 
where the union runs over the $\bf d$'s corresponding to the generators ${\bf t^d}$ of $\pp a^n$.  
Since  $\ell(R/\pp a^n)$ is finite, $\angle\, \backslash L_n$ is a bounded semi-open region. 
Hence there are only finitely many lattice points in  $\angle\, \backslash L_n$. 
Furthermore, as the number of generators of $\pp a^n$ increases as the power $n$ increases, 
so does the number of vertices of $L_n$. 
Now let $L$ be the convex hull containing all the $\frac{1}{n}L_n := \bigcup ( \frac{1}{n}\bf d + \angle) $ for $n \geq 1$. 
One sees that ${\HS(n)}$ is equal to the number of points in 
$ (\angle \,\backslash \frac{1}{n}L_n) \cap \frac{1}{n} \mathbb Z^d $ and a lattice cube in $\frac{1}{n} \mathbb Z^d $ has 
volume exactly $\frac{1}{n^d}$. Thus 
multiplying $\frac{1}{n^d}$ to $\HS(n)$ gives the Riemann sum that approximates the volume of $\angle \, \backslash L$ 
so that  when $n \rightarrow \infty$, the quantity $\frac{\HS(n)}{n^d}$ converges to the volume of $\angle \, \backslash L$. 
Hence the Hilbert-Samuel multiplicity $i(R, \pp a)$ of $R$ with respect to $\pp a$ 
can be computed using the following volume formula  ({\em c.f.} \cite[Exercise~2.8]{R1})
\begin{equation}\label{HSmulti}
i(R, \pp a) = {d!} \lim_{n \rightarrow \infty} \frac{\HS(n) }{n^d} = d! \, \mathcal Vol (\angle \, \backslash L) . 
\end{equation}

Now we assume that $\charact \kappa = p > 0$ and consider the Frobenius powers of the maximal ideal $\pp m$ 
consisting of all monomials in $R$ other than 1. 
(Similar considerations apply to arbitrary ideals of finite colength.)
Suppose $\pp m$ is generated by ${\bf t}^{{\bf d}_1}, \dots, {\bf t}^{{\bf d}_h}$. 
Then $\pp m^{ [p^e] }$ is generated by $({\bf t}^{{\bf d}_1})^{p^e} , \dots, ({\bf t}^{{\bf d}_h})^{p^e}$.
Let $L = \bigcup_{j=1}^h ({\bf d}_j+\angle)$ be the union of $h$ polyhedral cones. Then 
\[ \ell(R/\pp m^{[p^e]}) = \dim_{\kappa} R/\pp m^{[p^e]} 
  = \# \left ( p^e \cdot ( \angle \, \backslash L)  \bigcap \mathbb Z^d \right ).
 \]
This gives the Hilbert-Kunz function a combinatorial interpretation. More precisely, if we let 
$\mathcal P = \angle \backslash L $.
Then $\mathcal P$ is a semi-open polycell that does not contain its ``upper right" boundary faces. 
In particular, $\mathcal P$ is not necessarily (in fact, almost never) convex and 
$p^e \cdot \mathcal P \cap \mathbb Z^d$ consists of all the monomials in $R$ not in $\pp m^{[p^e]}$. 
The monomials that generate $R/\pp m^{[p^e]}$ as a $\kappa$ vector space are exactly those whose exponents
 are contained in $p^e\cdot \mathcal P$, the polycell $\mathcal P$ dilated by $p^e$.
Hence the values of the Hilbert-Kunz function at $e$ is 
\begin{equation}\label{HKEhrhart}
 \hk{R}{e} = \dim_k R/\pp m^{[p^e]} = \# ( p^e \cdot \mathcal P \cap \mathbb Z^d ). 
\end{equation}

To illustrate the ideas above, 
we use $R= k[s, st, st^2]$ from Example~\ref{ex_BG}. Then $\hk{R}{e} =\ell( R/\pp m^{[p^e] })$ can be
obtained by counting the number of lattice points belonging to the $p^e$-dilated semi-open polycell as shown in Figure~3 below.  

\begin{center} 
  \begin{tabular}{ccc} 
     \epsfig{file=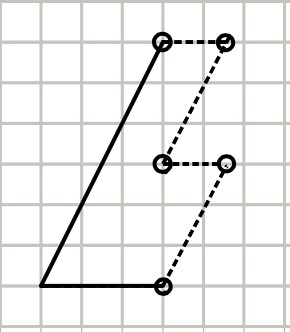, height=.5in} & \hspace{1in} & 
     \epsfig{file=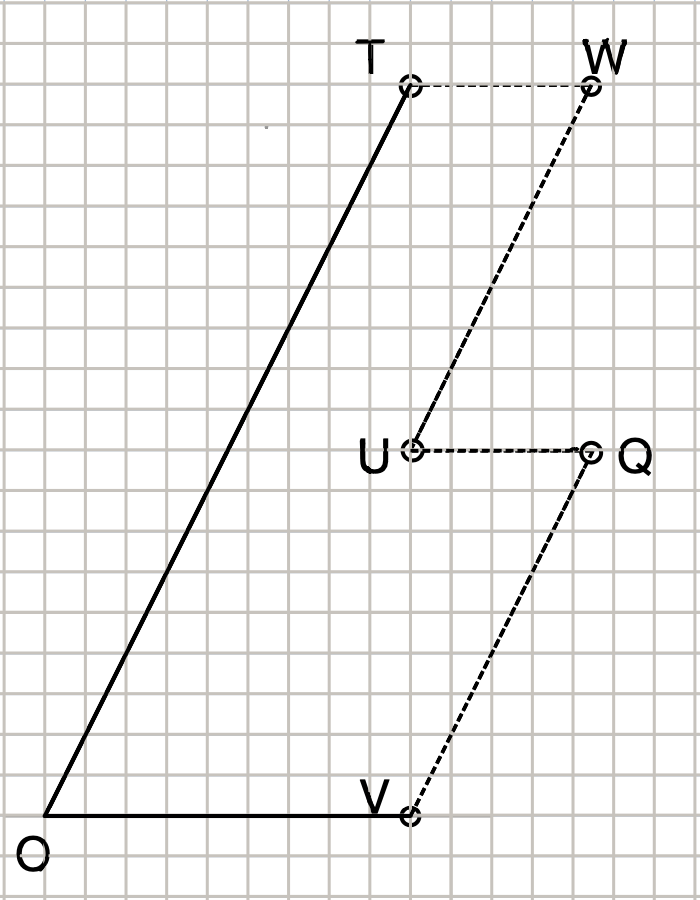, height=1.5in} \\
     {\em {\small $p=3, e=1$. }} & \hspace{1in} &  {\small \em $p=3, e=2$.}  
     \end{tabular} \\  
   \vspace{3mm}

{\small {\em Figure}~3: $p^e \cdot \mathcal P \cap \mathbb Z^2 $. (The labels will be needed in Example~\ref{regcone}.)}
\end{center} 
A complete computation of the Hilbert-Kunz function of this example will be presented in Example~\ref{regcone}.

In general, it is not difficult to see that 
\begin{equation}\label{shrinkP}
\#  \left( p^e \cdot \mathcal P \cap \mathbb Z^d \right ) = 
  \# \left( \mathcal P \cap \frac{1}{p^e} \cdot \mathbb Z^d \right ). 
\end{equation}
Hence similarly to the Hilbert-Samuel multiplicity, one can argue that the limit of 
$\hk{R}{e} $ divided by $(p^e)^d$ tends to the volume of $\mathcal P$. 
This is a geometric illustration of how Watanabe~\cite{Wat} proved the rationality of Hilbert-Kunz multiplicity for affine semigroup rings, 
although his proof does not involve Ehrhart theory. 

It is important to note that the coordinates of the vertices of $\mathcal P$ are not necessarily all integers. 
Therefore as $e$ varies, we are examining about the values of generalized Ehrhart functions which are eventually periodic with predictable period by Theorem~\ref{genEhrhart}. (See Remark~\ref{rmk_example}(a) for comments on periods.)

Next we reflect on the discussion above on Ehrhart theory, Hilbert-Samuel and Hilbert-Kunz functions. 
The value $\HS(n)$ of the Hilbert-Samuel function is the number of  lattice points in $\angle \backslash L_n$.  
Recall that $L_n= \bigcup ({\bf d} + \angle)$ where the union runs over the monomial generators ${\bf t}^{\bf d}$ of $\pp a^n$. 
The vertices of $\angle \backslash L_n$ are integral but its shape changes as $n$ increases 
due to the increasing number of generators. 
This prevents us from taking the advantage of the Ehrhart's Theorem. On the other hand, when working with the Frobenius powers of $\pp m$, the shape of  $p^e \cdot \mathcal P$ remains rigid and is always similar to $\mathcal P$. Each value of the Hilbert-Kunz function $\hk{R}{e}$ is thus exactly the value provided by the Ehrhart function $E_{\mathcal P}(p^e)$
for the $p^e$-th dilation of $\mathcal P$. This allows us to apply the generalized Ehrhart's Theorem~\ref{genEhrhart} for the rational polycell to its full potential. 
Hence 
\[ \hk{R}{e} =E_{\mathcal P}(p^e) , \text{ or equivalently,} \,\,\, \hkq{R}{q} = E_{\mathcal P}(q). \]
We then conclude that $\hkq{R}{q} $ is a quasipolynomial in $q$ 
and its value is given by the lattice points in a dilated polycell.

We now return to the BG decomposition. 
We see from (\ref{shrinkP}) that to compute $E_{\mathcal P}(p^e)$, one may restrict the BG decomposition 
$\Gamma$ to $\mathcal P$ and decompose $\mathcal P$ into a union of convex polycells. 
In this way, Bruns refined the generalized Ehrhart theorem in the case of affine semigroup rings 
and proved that the second coefficient is also a constant. 

In summary, the discussions above lead to the following identity for $e \gg 1$
\begin{equation}\label{twoHK}
\sum_{\bar \gamma} \mu_R(\gamma) \nu_{C_{\gamma}}(q) 
  = E_{\mathcal P}( q) 
\end{equation}
where the sum runs over the full-dimensional cells in $\bar \Gamma$. 
More precisely, (\ref{twoHK}) gives us two equivalent expressions for  $\hkq{R}{q}$.
As a corollary to Theorem~\ref{BGhk}, Theorem~\ref{fcnBruns} below refines Theorem~\ref{genEhrhart} by identifying that the second coefficient is always a constant for the case of normal affine semigroup rings. 

\begin{theorem}[Bruns~\cite{Bru1}]\label{fcnBruns}
The Hilbert-Kunz function of $R$ with respect to $\pp m$ is a quasi-polynomial. Precisely it has the form 
\[ \hkq{R}{q} = \mathcal{V}ol (\mathcal P ) q^d + \beta q^{d-1} + (\text{a quasipolynomial  in lower degrees}),  \] 
where $\beta$ is a constant, and $\mathcal{V}ol (\mathcal P)$ denotes the volume of the polycell $\mathcal P$ 
determined by the shape of $M$ in $\mathbb Z^d$. 
\end{theorem}  

Roughly speaking, by Theorem~\ref{BGhk}, the leading coefficient is 
$e_{HK} = \sum_{C_{\gamma}}  \mu_R(\gamma) \cdot   \mathcal{V}ol ( C_{\gamma} ) = \mathcal{V}ol (\mathcal P ) $,
and the second coefficient 
$\beta = \sum_{\bar{\gamma} \in \bar{\Gamma} }  \mu_R(\gamma) \cdot b_{\gamma} $ is a constant as expected 
where $b_{\gamma}$ is the second coefficient in $\nu_{\gamma}(n)$.

We now elaborate on the content of (\ref{twoHK}). First recall the ring $R = \kappa[s,st,st^2]$ from Example~\ref{ex_BG} and
the BG decomposition of the polycell $\mathcal P$: 
\begin{center} 
 \begin{tabular}{c} 
     \epsfig{file=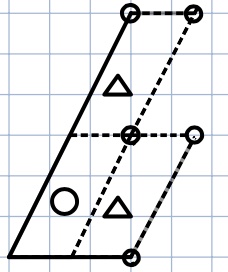, height=.8in}  \\
        {\small {\em Figure} 4: $\mathcal P$ by $\overline \Gamma$ } 
 \end{tabular}    
\end{center}
Thus (\ref{twoHK}) shows that
\[  E_{\mathcal P}( p^e) = \hk{R}{e}  = 
\mu_R(C_{ \gamma_1 }) \, \nu_{\gamma_1}(p^e)  +  \mu_R(C_{ \gamma_1 }) \,\nu_{\gamma_1}(p^e)
 = \nu_{\gamma_1}(p^e)  +  2 \, \nu_{\gamma_1}(p^e) \]
where $\gamma_1$ is equivalent to the cell labeled by $\bigcirc$ that occurs once and 
$\gamma_2$ by $\triangle$ that occurs twice. Moreover, each
$ \nu_{\gamma_i}(p^e) = E_{\lceil \gamma_i \rceil}(p^e)$ is an Ehrhart quasipolynomial of $\lceil \gamma_i \rceil$.

The following remarks point out some subtleties in (\ref{twoHK}). 

\begin{remark}\label{rmk_BG}
$(a)$ Despite the fact that  $\nu_{\gamma}(q)  =  E_{ \lceil \gamma \rceil} (q) $, 
a general one-to-one correspondence between 
the set of points counted by the left hand side in (\ref{twoHK}) and the actual lattice points in $\mathcal P$, counted by $E_{\mathcal P}( q)$, is not immediately obvious for arbitrary $\mathcal P$. 
One can carefully match them in the above Example~\ref{ex_BG}. 
Those in the interiors (of $\gamma_i$) are trivial, but the correspondence at the boundaries 
is much more subtle, especially in higher dimensions. 
This is the place where the BG decomposition lends us power, by 
relating the lattice points to those enclosed by the closure of the full-dimensional cells 
without doing the actual correspondence. \\
$(b)$
Since $\mathcal P$ is tessellated by the full-dimensional equivalence classes $\bar \gamma$, 
(\ref{twoHK}) suggests that the minimal generators $\mu_R(C_{\gamma})$ can be obtained by 
the number of times that $\bar \gamma$ appears in $\mathcal P$. This is clear in Example~\ref{ex_BG} (see {\em Figure}~4), 
but full generality would probably require some careful consideration. 
\end{remark}

We end the section by summarizing several families of rings whose Hilbert-Kunz function enjoys the following functional form 
with constants $\alpha$ and $\beta$,
\[ \hkq{R}{q} = \alpha q^d + \beta q^{d-1} + \gamma(q). \]
\begin{enumerate} 
\item $R$ is an excellent normal local domain or an excellent local ring satisfying (R1$'$). In this case, $\gamma(q) = O(q^{d-2})$.  (See Subsection~\ref{sub_classgp}.)
\item $R$ is a two-dimensional normal domain that is a standard graded algebra over an algebraically closed field $\kappa$. 
Indeed, $\beta=0$ in this case. Moveover, $\gamma(q)$ is a bounded function in general, and is eventually periodic if $\kappa$ is the algebraic closure of a finite field. (See Subsection~\ref{sub_sheaf}.)
\item $R$ is a normal affine semigroup ring. Then, $\gamma(q)$ is eventually a quasipolynomial. (See Theorem~\ref{fcnBruns}.) 
\item $R = \kappa[X]/I_2(X)$ where $I_2(X)$ is the ideal generated by all $2 \times 2$ minors of the $m\times n$ matrix $X$ of indeterminates. In this case, $\hkq{R}{q}$ is a polynomial. (See Subsection~\ref{sub_comb}.) 
\end{enumerate}

\section{Examples in Normal Affine Semigroup Rings}\label{example} 

In this section, we present several examples where the Hilbert-Kunz functions are obtained as generalized Erhart quasipolynomials 
by counting the lattice points inside the semi-open rational polycell $\mathcal P$ as described in Section~\ref{affineEhrhart}. 
First we count the points for the family of rational normal cones of arbitrary degree in the most 
intuitive way. Then we use another 2-dimensional example to illustrate  a {\tt Macaulay2} simulation. 
We also present some results and discussions on the homogeneous coordinate ring of Hirzebruch surfaces and its relative  
variations. The calculations presented here are not performed with  BG decompositions. 
However, we provide some comments for those who are interested in pursuing  this direction. 
Throughout the section, we assume $e\geq 1$.

\begin{example}\label{regcone} 
 Consider $\widehat{C}_g =\spec R_g \subset \mathbb C^{g+1}$, the rational normal cone of degree
  $g$ in characteristic $p$: $R_g=k[s, st, \cdots, st^g]$ or $ k[x_0,
  \cdots , x_g]/ J$ where $J$ denotes the ideal generated by the
  maximal minors of the matrix $\left(
\begin{matrix}x_0  & x_1 & \cdots & x_{g-1} \\ x_1 & x_2 & \cdots & x_g
\end{matrix} \right ) $.  The graph in the previous Figure ~3 illustrates $\widehat{C}_2$.
The Hilbert-Kunz function of $R_g$ with respect to the maximal graded ideal is
\[ \hk{R_g}{e} = (\frac{g+1}{2}) (p^e)^2 + \frac{1}{2}(-v_e^2+v_e g - g +1) \]
where 
 $v_e \in \{0, \dots, g-1\}$ is the congruence class of $p^e -1$ modulo $g$. 

We present a computation by an elementary approach. Observe the right graph for $e=2$ in Figure 3. 
The number of the lattice points in 
$p^2 \cdot \mathcal P \cap \mathbb Z^2$ can be obtained by counting those in $\bigtriangleup OTV$ 
without the right vertical boundary and the smaller regions $\bigtriangleup TUW$, 
$\bigtriangleup UVQ$ without the upper and right boundaries. For arbitrary $g$, 
there will be exactly $g$ identical smaller triangular shapes. 
The count in  $\bigtriangleup OTV$ denoted by $A_e$, is given by 
\[ \begin{array}{lcl} 
A_e & = & 1 + ( 1+ g ) + ( 1 + 2 g) + \cdots + (1 + (p^e-1) g)  \\
 & = & p^e + g \, [ 1 + 2 + \cdots + ( p^e-1 ) ] \\
 & = & \frac{g}{2} (p^e)^2 +  ( 1 - \frac{g}{2 } ) ( p^e ).
\end{array} \]
The number of the lattice points contained in the identical smaller triangles $\bigtriangleup TUW$ and $\bigtriangleup UVQ$,
denoted by $B_e$, is 
\[ \begin{array}{lcl} 
B_e & = & (p^e-1) + ( p^e -1-g) + (p^e -1 - 2g) + \cdots + (p^e -1 - u g)
\end{array} \] 
where $u$  is a non-negative integer such that $p^e -1 = u g + v_e$ with $v_e$ as described above.
After replacing $u$ by $(p^e - 1 -v_e)/g$, $B_e$ is simplified to be
\[ \begin{array}{lcl} 
B_e 
& = & (u+1) (p^e -1) - g \cdot \frac{(1+u)u }{2} \\ 
& = & \frac{1}{2g}  (p^e)^2 + ( \frac{1}{2} - \frac{1}{g} ) (p^e) - \frac{1}{2g} (v_e^2 - v_e g + g -1). 
\end{array} \] 
Thus the Hilbert-Kunz function of $R_g$ can be expressed as 
\[ \begin{array}{ll}
\hk{R}{e} & = A_e + g\cdot B_e  \\
& =  \frac{g}{2} (p^e)^2 +  ( 1 - \frac{g}{2 } ) ( p^e ) + 
g [ \frac{1}{2g}  (p^e)^2 + ( \frac{1}{2} - \frac{1}{g} ) (p^e) - \frac{1}{2g} (v_e^2 - v_e g + g -1) ] \\ 
& = \frac{1+g}{2} (p^e)^2  + 0 \cdot (p^e) -  \frac{1}{2} (v_e^2 - v_e g + g -1)  
 \end{array} \]
 with $v_e \equiv p^e -1 \pmod g $ as described.
\end{example}

It is a good exercise to carry out the computation for $\hk{R_g}{e}$ by the BG decomposition. 
There will be $g$ non-isomorphic conic divisors and each corresponds to a $2$-dimensional semiopen polycell $\gamma$. 
In fact, all of them are parallelograms of the same shape and size and for $i = 1, \dots, g$, there exists
exactly one cell appearing $i$ times in $\mathcal P$. 
However, since the polytopes are rational, polycells presenting non-isomorphic conic divisors do not necessarily have the same Ehrhart quasipolynomials. This adds complications but interesting twists to this exercise which leads to a better understanding of BG decompositions and Ehrhart quasipolynomials. 

\begin{example}\label{toricnonproj} 
$R=\kappa[s^2t, st, st^2]$  with ${\rm char}\, \kappa = p$. 
In this example, we describe an interpolation process using the values of $\hk{R}{e}$ provided by {\tt Macaulay2}.
In this example for $p=2$ and $e=1$, we use the following code: 
\begin{quote} 
{\tt 
p=2 \\
k=ZZ/p\\
T=k[s,t] \\
S=k[x,y,z] \\
f=map(T, R, $\{$s\^{}2*t, s*t, s*t\^{}2$\}$) \\
J=ker f \\
R=S/J \\
e=1 \\
q=p\^{}e \\ 
Ie=ideal(x\^{}q, y\^{}q, z\^{}q) \\
degree(R/Ie) \\
toRR(200, oo/4\^{}e)
}
\end{quote}
The second to last command produces the values of $\hk{R}{e}$ while 
the last command {\tt toRR(200, oo/4\^{}e)} divides the value by $(2^e)^2$ and converts it to 
the precision of 200 decimals. 
By Theorem~\ref{fcnBruns}, we know that $\hk{R}{e}$ must be in the form of 
\[ \hk{R}{e} = {\tt a} \, (2^e)^2 + {\tt b}\, (2^e) + c_e \] 
with rational constants {\tt a} and {\tt b}, and eventually periodic function $c_e$. Thus the sequence of outputs obtained by iterating the codes above with increasing $e$ must converge to the desired leading coefficient {\tt a}.
Since {\tt a} is a rational number, the sequence will eventually stabilize with apparent repeating decimals and 
one can guess an expression for {\tt a} as a quotient of two integers. 

Once the precise value {\tt a} is available, we can calculate {\tt b} by iterating similar {\tt Macaulay2} codes
with the last two input commands changed to 
\begin{quote}
{\tt 
degree(R/Ie) - a * q\^{}2 \\  
toRR(200, oo/2\^{}e)
}
\end{quote}
We do so until the sequence of outputs stabilizes with apparent repeating decimals.
Then do the simple computation to obtain {\tt b} in the form of a rational number. 
For the final term, do the {\tt Macaulay2} iterations with the last two input commands changed to
\begin{quote}
{\tt 
degree(R/Ie) - a * q\^{}2 - b * q \\  
toRR(200, oo) 
}
\end{quote}
In this step, the data may form multiple groups and each group will converge to a rational number.
In higher dimensional cases, we expect such a property will hold for the coefficients corresponding to 
terms of degree equal to $\dim R-2$. But for the remaining terms, the trend of the data might not be so clear. 

We perform this interpolation for small primes $p = 2, 3, 5, 7, 11$. The numerical outputs suggest the following results
\begin{equation}\label{ex2} \hkq{R}{q} = \frac{5}{3} q^2 - \frac{2}{3},  \,\,\,\,  \mbox{ if }  p \neq 3;  
\,\,\,\, \mbox{ and } \,\,\,\,
 \hkq{R}{q}= \frac{5}{3} q^2,  \,\,\,\,  \mbox{ if }  p = 3 . \end{equation}
By doing the lattice point count as done in Example~\ref{regcone}, one can show that 
the Hilbert-Kunz functions of $R$ for arbitrary primes are indeed given by (\ref{ex2}) as suggested by {\tt Macaulay2}. 
The BG decomposition can be performed in a similar way as commented in Example~\ref{regcone}.  
We leave these exercises to the interested readers. 
\end{example} 

We make two remarks regarding the above two examples. 
\begin{remark}\label{rmk_example}
$(a)$ In Example~\ref{regcone}, $\hk{R_g}{e}$ is a polynomial if and only
  if the class of $p^e -1$ modulo $g$ is independent of $e$. For instance, $g|p-1$ or $g=p$.
  In Example~\ref{toricnonproj}, $\hk{R}{e}$ is a polynomial for any given prime $p$
  even though the vertices of the polycell $\mathcal P$ defined in Section~\ref{affineEhrhart} 
  are obviously rational, but not integral in this example.
  
By the generalized Ehrhart Theorem~\ref{genEhrhart}, the period of an  Ehrhart quasipolynomial $E_{\mathcal P}(n)$ divides 
the least common multiple of the denominators of the vertices. Since Hilbert-Kunz function is specialized only at the powers of a prime $p$, the period may be smaller or even 1, but does not necessarily divide the least common multiple. 
These are clearly displayed in the previous two examples and the next one. For instance, in Example~\ref{regcone}, 
taking $p=3$ and $g=5$, then the least common multiple of the vertices is $5$ but the period of the Hilbert-Kunz function is $4$.

$(b)$ In both Examples~\ref{regcone} and \ref{toricnonproj}, the second coefficient $\beta$ vanishes. This is not a
  coincidence because a simplicial cone defines a
  quotient singularity which can be realized as a polynomial invariant
  of a finite group.  Therefore, its divisor group is a torsion group and
  hence $\beta =0$ ({\em c.f.} \cite{HMM},\cite{Ben},\cite[2.2]{F2},\cite{K5}).
\end{remark}

Next, suggested by U. Walther, we consider the family of Hirzebruch surfaces $X_{\Sigma_a}$.
Let $P_a$ be the integral polytope with vertices: $(0,0), (1,0), (0,1), (1,a+1)$ in $\mathbb R^2$. 
A {\em Hirzebruch surface} is a toric variety constructed from the normal fan 
$\Sigma_a$ associated to $P_a$.  One can also view $X_{\Sigma_a}$ 
as the Zariski closure of the set
$ \{ (1, s, t, st, \dots, st^{a+1}) | (s,t) \in (\mathbb C^*)^2 \}$ in $\mathbb P^{a+2}$. 
Notice that the powers of the monomials in $s, t$ are exactly the lattice points in $P_a$
 ({\em c.f.} \cite[Example~2.3.16]{CLS}). 
The ring $R=\kappa[M]$ in Example~\ref{Fa1} is the coordinate ring of the affine cone of $X_{\Sigma_a}$
with $M = \{ ( 1, {\bf v}) | {\bf v} \in P_a \cap \mathbb Z^2\}$. 

\begin{example}\label{Fa1}  
Set $\kappa = \mathbb Z /p$. Let $S_a$ be the affine semigroup ring arising from $X_{\Sigma_a}$ 
\[S_ a= \kappa[u, us, ut, ust, ust^2, \dots, ust^{a+1}]. \] 

We compute the Hilbert-Kunz function of $S_a$ with respect to the maximal ideal
generated by all monomials other than 1. By performing the same calculations using {\tt Macaulay2} as done in Example~\ref{toricnonproj}, the simulations suggest
\begin{enumerate}
\item $a=1$, 
\[ \begin{array}[t]{ll} 
\mbox{ if } p =2, & \hkq{S_1}{q} = \frac{7}{4}q^3 -\frac{1}{8} q^2 - \frac{1}{4} q ; \\ 
\mbox{ if } p \geq 3, & \hkq{S_1}{q} = \frac{7}{4}q^3 -\frac{1}{8} q^2 - \frac{1}{4} q -\frac{3}{8} .
\end{array}  \]

\item $a=2$, 
\[ \begin{array}{ll}

\mbox{ if } $p=3$, & \hkq{S_2}{q} = \frac{20}{9} q^3 - \frac{1}{3} q^2;  \\ 

\mbox{ if } p \neq 3, & \hkq{S_2}{q} = \frac{20}{9} q^3 - \frac{1}{3} q^2 + \left \{ 
   \begin{array}{ll} 
     - \frac{4}{9}, & q \equiv 2 \pmod 3 \\
     - \frac{8}{9}, & q \equiv 1 \pmod 3 
   \end{array} \right. .
\end{array}  \]
  
\end{enumerate} 
More cases are documented in Schalk~\cite{Sch2013}. 
The knowledge about the possible values of the period from Theorem~\ref{genEhrhart} is very helpful for the interpolation process of the quasipolynomials. 
The vertices of the polycell within which we count the lattice points must be the intersection of those parallel to any
three support hyperplanes of the affine cone of the semigroup generating $R$. By straightforward calculation, we know that the least multiple of the denominator of these vertices is $a+1$.
Thus the period of $\hkq{S_a}{q}$ is bounded above by $a+1$, although not necessarily dividing $a+1$, as indicated in Remark~\ref{rmk_example}(a). 

\end{example}

The Hilbert-Kunz function of Hirzebruch surfaces has also been considered via the pair 
$(X, \mathcal L)$ by Trivedi~\cite{Tr2016} using sheaf theoretic method where 
$X$ is regarded as a nonsingular ruled space over $\mathbb P^1$ and $\mathcal L$ is an ample line bundle on $X$. 
A closed form of the Hilbert-Kunz function is given in terms of formulas in the data provided by the ample line bundle $\mathcal L$. 

Saikali~\cite{Sai2018}  considered a variation of $S_a$ where the semigroup ring is a $k$-algebra generated by the monomials contained in the cone over the 
polytope $P_{b,h}$ with the vertices $(0,0,1)$, $(b, 0, 1)$, $(0, h, 1)$ and $(b+h, h, 1)$. 
Setting $b=h=1$ gives the special case in Example~\ref{Fa1} with $a=1$ up to an isomorphism, 
{i.e.,} $(P_a,1)=P_{b,h}$ when $a=b=h=1$. 
For $S_a$ in Example~\ref{Fa1}, all the nonzero lattice points are on the boundary of  $(1, P_a)$ for any $a$ while 
$P_{b,h}$ has lattice points in the interior if $h >1$. 
So the shape of the corresponding polycell $\mathcal P$ is a lot more complicated in \cite{Sai2018} in which
the lattice points are counted by viewing $\mathcal P$ as the union of a  large  convex polytope and multiple prickly smaller shapes similar to the way Example~\ref{regcone} is considered. A closed form of the Hilbert-Kunz function is obtained by analyzing complicated counting functions  with the assistance of the 3D graphing software {\sc Geogebra}.   
 
The BG decompositions for $S_a$ in Example~\ref{Fa1} and for the variations in \cite{Sai2018} are much more subtle 
since these are nonsimplicial three-dimensional polycells and they have not been explored in great detail. 

We observe the following general facts. Let $R$ be a normal afffine semigroup ring and $R_+$ be the maximal ideal consisting of all the monomials other than 1. 

\begin{remark}\label{construction}  
The Hilbert-Kunz multiplicity of an affine semigroup ring, as it can be expressed by the volume of $\mathcal P$, 
is independent of the characteristic $p$.  
Obviously the Hilbert-Kunz function is not. 
Not only does the shape of $\hk{R}{e}$ depend on $p$, but,  even in polynomial form, the function varies as $p$ does.

For a given $S_a$,  if $q$ is constant modulo $ {a+1}$, then $\hkq{R}{q}$ is in a polynomial form. 
Otherwise, it is likely to be a quasipolynomial. 
This indicates that the construction of a normal semigroup ring with a desired form of Hilbert-Kunz function 
is, in theory, plausible. 
In practice though, in spite of Theorem~\ref{genEhrhart} that identifies the possible period  of the Ehrhart quasipolynomial from
the coordinates of the vertices, the exact period is not immediately obvious (see references following Question~\ref{period}).  

Furthermore, since the Ehrhart quasipolynomial of a polytope with integral vertices is a polynomial, it is possible to construct affine semigroup rings whose Hilbert-Kunz functions are true polynomials. In fact, this happens if all the cells of $\Gamma$ have integral vertices. Burns pointed out in \cite[p.~71]{Bru1} that this is possible if the affine semigroup $M$ has a unimodular configuration.
As we can see from the examples above, the integral condition on the full dimensional cells are not necessary 
for Hilbert-Kunz functions to be of true polynomial forms. While the configuration of semigroups are independent of the characteristic, the functional form of Hilbert-Kunz functions in general does depend on the characteristic.
\end{remark}

\begin{remark}\label{shape}
The open cells in $\overline \Gamma$ of the BG decomposition determine the ultimate shape of $\hk{R}{e}$. 
We say that a function $F(e)$ is a polynomial up to a periodic function if it is in the form of  
\[ F(e) =  a_{d} (p^e)^{d} + \cdots + a_1 (p^e) + \delta_e \]
where $a_d, \dots, a_1$ are constants and $\delta_e$ is a periodic function. 
All the functions presented in this section are in such a form. Next we call  $\hk{R}{e+1} - \hk{R}{e}$ 
  the {\em difference function} of $\hk{R}{e} $.
 By a straightforward computation, we can show that if  the difference function is a 
 polynomial up to a periodic function, then so is $\hk{R}{e}$.  
This is supported by Saikali's analysis in \cite{Sai2018}. 
It is also related to the shape of the full dimensional cells in the BG decomposition. 
Not all full-dimensional cells have its Ehrhart quasipolynomial as a polynomial up to a periodic function. 
It is reasonable to expect that the Hilbert-Kunz functions can be characterized by the BG decomposition. In another words, 
we expect that the functional form of Hilbert-Kunz function can be determined by the combinatorial structure of the affine semigroup ring. 
\end{remark}

By studying the BG decompositions, one might be able to understand better the periodic behavior of $\hk{R}{e} $ as pointed out in Remarks~\ref{construction} and \ref{shape}. In addition, we may challenge ourselves to find a geometric interpretation of the coefficients of the lower degree terms, and search for useful ways to determine the period of the quasipolynomial. 

We end this paper with the following questions. An answer to Question~\ref{computation} will provide an effective algorithm of accessing the Hilbert-Kunz functions of normal affine semigroup rings.

\begin{question}\label{coeff} 
As in Theorem~\ref{Ehrhart} for the leading and the second coefficients,  
can one give a meaningful geometric interpretation to 
the coefficient of each degree in the polynomial form of $\hk{R}{e}$? 
\end{question} 

\begin{question}\label{period} 
How can the period of $\hk{R}{e} $ be effectively and efficiently estimated? 

An Ehrhart version of Question~\ref{period}  is posted in Beck and Robins~\cite[3.39]{BeR}; see also 
Woods~\cite{Wo2005}, and Beck, Sam and Woods~\cite{BSW2008}. 
\end{question}

Let $R=k[M]$ be a normal affine semigroup ring in Subsection~\ref{sub_BG}. Note that the BG decomposition is determined by the semigroup structure of $M$. 
\begin{question}\label{computation}
Can an effective algorithm be developed for the function $\nu_{\gamma}(n)$ and the minimum generators 
$\mu_R(C_{\gamma})$  for any full dimensional class $\overline \gamma$ in $\overline \Gamma$? Furthermore, does the semigroup structure of $M$ characterize the BG decomposition and vice versa? 
\end{question}

Counting lattice points inside a polytope, though a simple and rather
primitive task, is not as easy as it appears. 
The desire of understanding Hilbert-Kunz function now leads us to ask for
what class of polytope $\Delta$ can  $E_{\Delta}(n) $ be explicitly calculated. 
Starting from an affine semigroup and using the BG decomposition, it seems that one may branch out into 
Hilbert-Kunz or Ehrhart theory. The results obtained from either direction may very likely enhance each other.

\bigskip

\noindent{\bf Acknowledgement.} My sincere gratitude goes to Kazuhiko Kurano 
for showing me the connection between Hilbert-Kunz functions and Ehrhart Theory. 
To Joseph Gubeladze for pointing to 
the beautiful source \cite{BeR}. To Uli Walther for 
discussions on the Macaulay 2 codes. To Holger Brenner and Claudia Miller for 
insightful conversations offering support and many crucial references. 
Special thanks to I-Chiau Huang and Academia Sinica of Taiwan who have sponsored multiple visits 
throughout the course while the main ideas in this article were formed. 
I am grateful for the anonymous referees for providing detailed comments and corrections that greatly improve
the manuscript. 


\medskip

{\small 
\noindent Department of Mathematics, Central Michigan University, Mt. Pleasant, MI~48859, U.S.A.

\noindent {\em email}: chan1cj@cmich.edu

}

\end{document}